\documentclass{amsart}

\usepackage{amsmath,amssymb,amsthm}
\usepackage{braket,mathtools,bbm}
\usepackage{amscd}

\usepackage{graphicx}
\usepackage[subrefformat=parens]{subcaption}

\newtheorem{lma}{Lemma}[section]
\newtheorem{prop}[lma]{Proposition}
\newtheorem{thm}[lma]{Theorem}
\newtheorem{cor}[lma]{Corollary}

\theoremstyle{definition}

\newtheorem{rmk}[lma]{Remark}

\numberwithin{equation}{section}

\newcommand{\defeq}{=}
\newcommand{\bbN}{\mathbb{N}}

\newcommand{\bbR}{\mathbb{R}}

\newcommand{\bbS}{\mathbb{S}}
\newcommand{\rmd}{\mathrm{d}}
\newcommand{\incmat}{\mathbb{M}}
\newcommand{\geopot}{u}
\newcommand{\steplenpot}{\psi}

\newcommand{\coding}{\pi}
\newcommand{\holder}{H\"older }
\newcommand{\recset}{\mathcal{R}}
\newcommand{\unifrecset}{\mathcal{UR}}
\newcommand{\windowword}{\mathcal{B}}

\newcommand{\hausdim}{\dim_H}

\newcommand{\diam}{\mathrm{diam}}

\newcommand{\lshift}{\sigma}

\newcommand{\gibbsmeas}{\nu}

\newcommand{\distconst}{V}

\newcommand{\shiftsp}{\Sigma}
\newcommand{\levelset}{\mathcal{L}}

\newcommand{\interval}{I}
\newcommand{\symbdist}{d}

\newcommand{\udim}{\dim_\geopot}
\newcommand{\fenerg}{\beta}
\newcommand{\attractor}{J}
\newcommand{\lodert}{\underline{D}}
\newcommand{\updert}{\overline{D}}
\newcommand{\cdf}{F}
\newcommand{\moderate}{\mathcal{UD}}

\newcommand{\contraction}{f}
\newcommand{\symbset}{A}

\newcommand{\postfixset}{\mathcal{S}}
\newcommand{\infixset}{\mathcal{I}}
\newcommand{\newsymbset}{\mathcal{A}}

\newcommand{\wordset}{\mathcal{W}}
\newcommand{\freqappearset}{\mathcal{F}}

\newcommand{\notfreq}{\freqappearset'}

\newcommand{\cylchain}{\mathcal{C}}

\newcommand{\vertexset}{\mathcal{V}}
\newcommand{\edgeset}{\mathcal{E}}
\newcommand{\initialv}{p_-}
\newcommand{\terminalv}{p_+}
\newcommand{\intervalfamily}{\mathcal{\interval}}
\newcommand{\interior}{\mathrm{Int}}
\newcommand{\uniflevset}{\mathcal{UL}}
\newcommand{\onedimmani}{\mathcal{M}^1}
\newcommand{\GDMS}{\Phi}

\newcommand{\steplenlo}[2]{\alpha^-_{#1,#2}}
\newcommand{\steplenup}[2]{\alpha^+_{#1,#2}}

\title[Uniform level sets with constraints on word appearance]{A multifractal analysis for uniform level sets with constraints on word appearance}
\author{Ziyu Liu}
\address{Graduate School of Mathmatics, Nagoya University}
\email{m19050e@math.nagoya-u.ac.jp}

\begin{document}
\begin{abstract}
    Given a topologically transitive subshift of finite type, the $\geopot$-dimension of the $\alpha$-level set of the quotient of Birkhoff sums is a well-studied topic. In this paper, we consider the uniform $\alpha$-level set, which is a subset of the $\alpha$-level set. We shall show that the $\alpha$-level set and the uniform $\alpha$-level set have the same $\geopot$-dimension. Moreover, we also consider two types of subsets of the $\alpha$-level set, whose elements are sequences satisfying some constraints on their subwords. We shall show that for $\alpha$ not on the boundary of the dimension spectrum, these subsets also have the same $\geopot$-dimension as the $\alpha$-level set. As an application, we shall perform a multifractal analysis of the \holder regularity of a Gibbs measure on a one-dimensional manifold.
\end{abstract}
\maketitle
\section{Introduction and Statement of Main Results}
    Let $\lshift:\shiftsp\rightarrow\shiftsp$ be a topologically transitive one-sided subshift of finite type (SFT). Let $v:\shiftsp\rightarrow(-\infty,0)$ and $\steplenpot:\shiftsp\rightarrow\bbR$ be two \holder continuous functions. The meaning of the \holder continuity of these functions will be made clear in Subsection \ref{subsec:metricshiftsp}.
		
    For a function $f:\shiftsp\rightarrow\bbR$ and a positive integer $n$, $S_nf\defeq\sum_{k=0}^{n-1}f\circ\lshift^k$ will be called the $n$-th Birkhoff sum of $f$. Define $S_0f$ to be the zero function. For any $\alpha\in\bbR$, the $\geopot$-dimension of the following $\alpha$-level set of  the quotient of Birkhoff sums
    \begin{equation*}
        \levelset^\alpha_{ \steplenpot, v}\defeq\Set{\xi\in\shiftsp|\lim_{n\rightarrow+\infty}\frac{S_n\steplenpot(\xi)}{S_nv(\xi)}=\alpha}
    \end{equation*}
    has been given by \cite{pesinweiss0,schmeling}. Here, $\geopot:\shiftsp\rightarrow(0,+\infty)$ is a \holder continuous function, and the $\geopot$-dimension is a notion introduced by Barreira and Schmeling \cite{barsch}. Similar to Hausdorff dimension, it is also defined through the Caratheodory's construction of measures; see \cite{barsch} for details. In fact, owing to the \holder continuity of $\geopot$, the $\geopot$-dimension $\udim(E)$ of an arbitrary $E\subseteq\shiftsp$ equals its Hausdorff dimension with metric taken to be $d_\geopot$; see Section \ref{sec:known} for the definition of the metric $d_\geopot$.
    
    We shall summarize in Subsection \ref{subsec:pre:gibbs} some of the results proved in \cite{pesinweiss,schmeling}. Postponing the details to Subsection \ref{subsec:pre:gibbs}, we here simply state the known results on the $\geopot$-dimension of $\levelset^\alpha_{\steplenpot,v}$. Define
    \begin{equation} \label{eqn:steplenloupdef}
		\steplenlo{\steplenpot}{v}\defeq\inf_{\xi\in\shiftsp}\liminf_{n\rightarrow+\infty}\frac{S_n\steplenpot(\xi)}{S_nv(\xi)};\quad\steplenup{\steplenpot}{v}\defeq\sup_{\xi\in\shiftsp}\limsup_{n\rightarrow+\infty}\frac{S_n\steplenpot(\xi)}{S_nv(\xi)}.
    \end{equation}
    It is known that $\levelset^\alpha_{\steplenpot,v}$ is non-empty if and only if $\steplenlo{\steplenpot}{v}\leq\alpha\leq\steplenup{\steplenpot}{v}$ \cite{schmeling}. For each $\alpha\in[\steplenlo{\steplenpot}{v},\steplenup{\steplenpot}{v}]$ and $q\in\bbR$, there exists a unique $\fenerg_\alpha(q)\in\bbR$ such that
    \begin{equation*}
        P(q(\steplenpot-\alpha v)-\fenerg_\alpha(q)\geopot)=0,
    \end{equation*}
    where $P(\cdot)$ denotes the topological pressure. The $\geopot$-dimension of $\levelset^\alpha_{\steplenpot,v}$ is then $\udim(\levelset^\alpha_{ \steplenpot, v})=\inf_{q\in\bbR} \fenerg_\alpha(q)$ \cite{pesinweiss}. The assertions above are all contained in Theorem \ref{thm:pw}.

    There are a number of articles, e.g. \cite{FS,GJK}, exploring the dimension of subsets of the level sets $\levelset^\alpha_{\steplenpot, v}$. Among these subsets, we are in particular interested in what we shall call the \emph{uniform level sets} of the quotient of Birkhoff sums, defined by
    \begin{align*}
        \uniflevset^\alpha_{ \steplenpot, v}=\Set{\xi\in\shiftsp|\sup_{n\in\bbN}|S_n\steplenpot(\xi)-\alpha S_nv(\xi)|<+\infty}.
    \end{align*}
    Clearly, $\uniflevset^\alpha_{ \steplenpot, v}\subseteq \levelset^\alpha_{\steplenpot, v}$ for any $\alpha\in\bbR$. Now that the multifractal analysis of Gibbs measures provides a formula of $\levelset^\alpha_{\steplenpot, v}$, one would naturally ask whether
    \begin{equation}\label{eqn:gjk}
        \udim(\uniflevset^\alpha_{ \steplenpot, v})=\udim(\levelset^\alpha_{ \steplenpot, v})
    \end{equation}
    always holds.
    
    Indeed, if one add some extra assumptions, this question has been answered affirmatively. Fan and Schmeling showed (\ref{eqn:gjk}) in \cite{FS}, when $\geopot$ is a constant function, $\lshift$ is a full shift and $\steplenlo{\steplenpot}{v}<\alpha<\steplenup{\steplenpot}{v}$. Later, in a recent article \cite{GJK} by Gr\"oger, Jaerisch and Kesseb\"ohmer, the authors essentially showed the same assertion when $\lshift$ is a full shift and $\steplenlo{\steplenpot}{v}<\alpha<\steplenup{\steplenpot}{v}$. In other words, the proof in \cite{GJK} does not need the function $\geopot:\shiftsp\rightarrow(0,+\infty)$ to be constant. In this paper, we shall show that (\ref{eqn:gjk}) holds generally. More precisely, we claim the following proposition.
    \begin{prop}\label{prop:main:rec}
        Let $\lshift:\shiftsp\rightarrow\shiftsp$ be a topologically transitive SFT and let the functions $\steplenpot:\shiftsp\rightarrow\bbR$, $v:\shiftsp\rightarrow(-\infty,0)$ and $\geopot:\shiftsp\rightarrow(0,+\infty)$ be \holder continuous. Then, for any $\alpha\in\bbR$, we have $\udim(\uniflevset^\alpha_{ \steplenpot, v})=\udim(\levelset^\alpha_{ \steplenpot, v})$.
    \end{prop}
	

    The case where $\alpha\in\set{\steplenlo{\steplenpot}{v},\steplenup{\steplenpot}{v}}$ will be called the \emph{boundary case}, and the case where $\steplenlo{\steplenpot}{v}<\alpha<\steplenup{\steplenpot}{v}$ will be called the \emph{interior case}. The arguments in \cite{GJK} are well suited for the interior case. This enables us to develop the arguments in \cite{GJK} so as to show that two types of subsets of $\uniflevset^\alpha_{\steplenpot,v}$ also have the same $\geopot$-dimension as $\levelset^\alpha_{\steplenpot,v}$ for the interior case. The precise statements are to be given in Theorem \ref{thm:main:freq} and Theorem \ref{thm:main:notfreq}, which are the main results of this paper.
		
    Let $\wordset$ be a finite set of admissible words. Define for each positive integer $k$,
    \begin{equation*}
        \freqappearset_{\wordset,k} \defeq\bigcap_{n=0}^\infty\set{\xi\in\shiftsp|\text{all the words in }\wordset\text{ are the subwords of }\xi_{n+1}\cdots\xi_{n+k}},
    \end{equation*}
    and then set $\freqappearset_\wordset\defeq\bigcup_{k=1}^\infty\freqappearset_{\wordset,k}$. In words, an element of $\freqappearset_\wordset$ is an admissible sequence in which all words from $\wordset$ appears regularly. With $\freqappearset_\wordset$ defined, we are now able to state our first main theorem.
	\begin{thm}\label{thm:main:freq}
        Let $\lshift:\shiftsp\rightarrow\shiftsp$ be a topologically transitive SFT and let the functions $\steplenpot:\shiftsp\rightarrow\bbR$, $v:\shiftsp\rightarrow(-\infty,0)$ and $\geopot:\shiftsp\rightarrow(0,+\infty)$ be \holder continuous. Suppose that $\wordset$ is a finite set of admissible words over $\symbset$. Then, for any $\alpha\notin\set{\steplenlo{\steplenpot}{v},\steplenup{\steplenpot}{v}}$, we have
        \begin{equation}
            \udim( \uniflevset^\alpha_{ \steplenpot,v} \cap\freqappearset_\wordset)=\udim(\levelset^\alpha_{\steplenpot,v}).
		\end{equation}
    \end{thm}
		
    In order to state the second main theorem, for every non-negative integer $n$, we define the $n$-th power of a word $\omega$ as $\omega^n\defeq\omega\cdots\omega$, where the right-hand side is the $n$-fold concatenation of $\omega$. More formally, define $\omega^0$ as the empty word, and for each positive integer $n$, define $\omega^n\defeq\omega^{n-1}\omega$. Define $\notfreq_\wordset\defeq\bigcup_{l=1}^\infty\notfreq_{\wordset,l}$, where for any positive integer $l$,
    \begin{equation*}
        \notfreq_{\wordset,l} \defeq\bigcap_{\omega\in\wordset}\Set{\xi\in\shiftsp|\xi\text{ does not contain }\omega^l\text{ as a subword}}.
    \end{equation*}
    In words, the sequences in $\notfreq_\wordset$ are those in which none of the words in $\wordset$ appears with arbitrarily high power. Our second main theorem then claims the following.
	\begin{thm}\label{thm:main:notfreq}
		Let $\lshift:\shiftsp\rightarrow\shiftsp$ be a topologically transitive SFT and let the functions $\steplenpot:\shiftsp\rightarrow\bbR$, $v:\shiftsp\rightarrow(-\infty,0)$ and $\geopot:\shiftsp\rightarrow(0,+\infty)$ be \holder continuous. Suppose that $\wordset$ is a finite set of admissible words over $\symbset$. Then, for any $\alpha\notin\set{\steplenlo{\steplenpot}{v},\steplenup{\steplenpot}{v}}$, we have
        \begin{equation}
            \udim( \uniflevset^\alpha_{ \steplenpot,v} \cap\notfreq_\wordset)=\udim(\levelset^\alpha_{\steplenpot,v}).
		\end{equation}
    \end{thm}
    \begin{rmk}
		Evidently, all the powers of an admissible word $\omega$ are admissible if and only if $\omega^2=\omega\omega$ is admissible. If we define $\tilde{\wordset}\defeq\set{\omega\in\wordset|\omega^2\text{ is admissible}}$, we will clearly have $\notfreq_{\tilde{\wordset}}=\notfreq_\wordset$. Hence, without loss of generality, when we check the validity of Theorem \ref{thm:main:notfreq}, we may assume that $\wordset$ is a (possibly empty) finite set of words whose second powers are all admissible.
    \end{rmk}


    We assumed that $v$ is a negative \holder continuous function in Proposition \ref{prop:main:rec}, Theorem \ref{thm:main:freq} and Theorem \ref{thm:main:notfreq}. In fact, it is not difficult to see that the assertions remain true if $v$ is positive.

    Theorem \ref{thm:main:freq} and Theorem \ref{thm:main:notfreq} have the following corollary, whose proof will be presented in Subsection \ref{subsec:cor:freqnotfreqmaxdim}.
    \begin{cor} \label{cor:freqnotfreqmaxdim}
        Let $\lshift:\shiftsp\rightarrow\shiftsp$ be a topologically transitive SFT and $\wordset$ be a finite set of admissible words. Then, we have $\udim(\freqappearset_\wordset)=\udim(\notfreq_\wordset)=\udim(\shiftsp)$.
    \end{cor}
    
    Both Theorem \ref{thm:main:freq} and Theorem \ref{thm:main:notfreq} are claimed only for the non-boundary case. As we shall see in Proposition \ref{prop:counterex}, Theorem \ref{thm:main:freq} fails in general for the boundary case. As for Theorem \ref{thm:main:notfreq}, it is not yet clear whether the claim always holds for the boundary case.

    As an application, we study the \holder regularity of a Gibbs measure on a one-dimensional manifold. A one-dimensional connected (second-countable) manifold without boundary must be either the unit circle $\bbS^1$ or the real line $\bbR$. In what follows, we will use $\onedimmani$ to denote either $\bbS^1$ or $\bbR$.
    
    Let $(\vertexset,\edgeset)$ be a directed graph, where $\vertexset$ is a finite set of vertices and $\edgeset$ is a finite set of edges. We assume that $\edgeset$ contains at least two edges. We allow the existence of multiple directed edges with the same initial and terminal vertices. We always assume that the graph $(\vertexset,\edgeset)$ is strongly connected, meaning that for any ordered pair of vertices $(p,p')$, there exists at least one path from $p$ to $p'$. Let $\intervalfamily=\set{\interval_p|p\in\vertexset}$ be a family of compact intervals of positive length in $\onedimmani$. To each directed edge $e\in\edgeset$, whose initial vertex and terminal vertex are $\initialv(e)$ and $\terminalv(e)$ respectively, we associate one contraction $\contraction_e:\interval_{\terminalv(e)}\rightarrow\interval_{\initialv(e)}$. Then, $\GDMS\defeq\set{\contraction_e|e\in\edgeset}$ is called a graph directed Markov system (GDMS) \cite{gdms}. We shall further impose the open set condition and a condition on the smoothness of these contractions in Subsection \ref{subsec:app:setup}. With these conditions assumed, the GDMS $\GDMS$ satisfies the axioms of a conformal graph directed Markov system (CGDMS) in \cite{gdms}. The definition of a CGDMS in a more general setup is given in \cite{gdms}.
    
    There is a natural SFT $\lshift:\shiftsp\rightarrow\shiftsp$ associated with the aforementioned CGDMS. The set of symbols is $\edgeset$, and the incidence matrix $\incmat:\edgeset\times\edgeset\rightarrow\set{0,1}$ is defined by setting $\incmat(e,e')=1$ if and only if $\terminalv(e)=\initialv(e')$, for any two edges $e,e'\in\edgeset$. The coding map $\coding:\shiftsp\rightarrow\bbS^1$ is defined by
    \begin{equation*}
        \set{\coding(\xi)}= \bigcap_{k=1}^\infty\contraction_{\xi_1}\cdots\contraction_{\xi_k}(\interval_{\terminalv(\xi_k)}),
    \end{equation*}
    for every $\xi\in\shiftsp$. The limit set of the CGDMS is $\attractor\defeq\coding(\shiftsp)$. A Gibbs measure $\gibbsmeas$ on $\shiftsp$ can thus be transferred to its pushforward measure $\coding_*\gibbsmeas$ on $\attractor$ through the coding map $\coding$. We shall call $\coding_*\gibbsmeas$ a Gibbs measure on $\attractor$.
    
    For each $e\in\edgeset$, define $\interval_e\defeq\contraction_e(\interval_{\terminalv(e)})\subseteq \interval_{\initialv(e)}$. By the open set condition and the assumption that $\edgeset$ contains at least two edges, we have $\interval_e\subsetneq\onedimmani$. Consequently, for any two $x,y\in\interval_e$, there exists a unique minimal subinterval of $\interval_e$ containing both $x$ and $y$, which will be denoted by $[x,y]_e$. Given a Gibbs measure supported on $\attractor$ and $\alpha\geq 0$, for any $e\in\edgeset$ and any $x$ in the interior of $\interval_e$, define
    \begin{align*}
        \lodert^\alpha \coding_*\gibbsmeas(x) &\defeq\lim_{\varepsilon\rightarrow 0}\inf\Set{\frac{\coding_*\gibbsmeas([x,y]_e)}{d_{\onedimmani}(x,y)^\alpha}|y\in\interval_e,\,d_{\onedimmani}(x,y)<\varepsilon};\\
        \updert^\alpha \coding_*\gibbsmeas(x) &\defeq\lim_{\varepsilon\rightarrow 0}\sup \Set{\frac{\coding_*\gibbsmeas([x,y]_e)}{d_{\onedimmani}(x,y)^\alpha}|y\in\interval_e,\,d_{\onedimmani}(x,y)<\varepsilon}.
    \end{align*}
    When $\onedimmani=\bbR$, $d_{\onedimmani}$ is the standard Euclidean metric; when $\onedimmani=\bbS^1$, $d_{\onedimmani}$ is the Euclidean metric on the unit circle. By the open set condition we assumed, if $x$ is in the interior of some $\interval_e$, then this $e$ is unique. Therefore, $\lodert^\alpha \coding_*\gibbsmeas(x)$ and $\updert^\alpha \coding_*\gibbsmeas(x)$ does not depend on the choice of $e$. The set to which we shall perform dimension analysis is
    \begin{equation*}
        \moderate^\alpha_\gibbsmeas \defeq\Set{x\in\onedimmani\setminus\bigcup_{e\in\edgeset}\partial\interval_e|0<\lodert^\alpha\coding_*\gibbsmeas(x)\leq\updert^\alpha\coding_*\gibbsmeas(x)<+\infty}.
    \end{equation*}
    Since $\edgeset$ is assumed to be a finite set, we have $\bigcup_{e\in\edgeset}\partial\interval_e$ is also a finite set. Hence, it is acceptable to exclude the points in $\bigcup_{e\in\edgeset}\partial\interval_e$ from a dimension-theoretic perspective.
    
    Note that for every $x\notin\attractor$, we have $\lodert^\alpha \coding_*\gibbsmeas(x)=\updert^\alpha \coding_*\gibbsmeas(x)=0$, because some neighbourhood of $x$ is disjoint from $\attractor$, and is thus a $\coding_*\gibbsmeas$-null set. Hence, we have
    \begin{equation*}
        \moderate^\alpha_\gibbsmeas \defeq\Set{x\in\attractor\setminus\bigcup_{e\in\edgeset}\partial\interval_e|0<\lodert^\alpha\coding_*\gibbsmeas(x)\leq\updert^\alpha\coding_*\gibbsmeas(x)<+\infty},
    \end{equation*}
    meaning that we may restrict our attention to $x\in\attractor$.
    
    When $\onedimmani=\bbR$, the measure $\coding^*\gibbsmeas$ is represented by its cumulative distribution function $\cdf_\gibbsmeas$. By definition, $\lodert^\alpha \coding_*\gibbsmeas(x)$ and $\updert^\alpha \coding_*\gibbsmeas(x)$ are precisely the lower and upper $\alpha$-\holder derivatives of $\cdf_\gibbsmeas$, which appear in \cite{kessestratholder}. For $\onedimmani=\bbS^1$, although the cumulative distribution function is not well-defined, yet we can still make a similar argument locally. Thus, we can see that $\moderate^\alpha_\gibbsmeas$ is a natural set to consider. The elements of $\moderate^\alpha_\gibbsmeas$ are the points $x$ for which $\coding_*\gibbsmeas([x,y]_e)$ changes neither too rapidly nor too slowly compared with $d_{\onedimmani}(x,y)^\alpha$ as $y$ varies in a small neighbourhood of $x$.
    
    Our result on the Hausdorff dimension of $\moderate^\alpha_\gibbsmeas$ is the following theorem.
    \begin{thm}\label{thm:app}
        Let $\GDMS\defeq\set{\contraction_e|e\in\edgeset}$ be a CGDMS satisfying all the conditions we assume in Subsection \ref{subsec:app:setup}. Let $\lshift:\shiftsp\rightarrow\shiftsp$ be the topologically transitive SFT associated with the CGDMS $\GDMS$. Let $\steplenpot:\shiftsp\rightarrow\bbR$ be a \holder continuous function satisfying $P(\steplenpot)=0$. Let $\gibbsmeas$ be a Gibbs measure of $\steplenpot$. Define $v:\shiftsp\rightarrow\bbR$ by
        \begin{equation} \label{eqn:geopotdef}
            v(\xi)\defeq \log|\contraction'_{\xi_1}(\coding(\xi))|<0
        \end{equation}
        for any $\xi\in\shiftsp$, and set $\geopot:\shiftsp\rightarrow(0,+\infty)$ to be $-v$. Then, for any $\alpha\notin \set{\steplenlo{\steplenpot}{v},\steplenup{\steplenpot}{v}}$,
        \begin{equation*}
            \hausdim( \moderate^\alpha_\gibbsmeas)=\udim(\levelset^\alpha_{\steplenpot,v}).
        \end{equation*}
    \end{thm}

    The key observation for showing Theorem \ref{thm:app} is the following inclusion in Proposition \ref{prop:der}:
    \begin{equation} \label{eqn:prop:der}
        \uniflevset^\alpha_{ \steplenpot,v}\cap\notfreq_\wordset\subseteq\coding^{-1}\left(\moderate^\alpha\right)\subseteq\uniflevset^\alpha_{\steplenpot,v},
    \end{equation}
    where $\alpha$ is an arbitrary real number, and $\wordset$ is a finite set of admissible words to be defined in Section \ref{sec:app}. This suggests that Theorem \ref{thm:main:notfreq} can be applied to prove Theorem \ref{thm:app}. Following this idea, we shall give a proof in Section \ref{sec:app}.
    
    Recall that the multifractal analysis of Gibbs measures shows that there exists a unique $\alpha_0\in[\steplenlo{\steplenpot}{v},\steplenup{\steplenpot}{v}]$ such that $\udim(\levelset^{\alpha_0}_{\steplenpot,v})=\udim(\shiftsp)$ \cite{pesinweiss}. Thus, this fact and Theorem \ref{thm:app} suggest the following claim.
    \begin{cor}\label{cor:maxdim}
        Assume that $\hausdim(\attractor)>0$. Then, there is a unique $\alpha_0\in\bbR$ such that $\hausdim( \moderate^{\alpha_0}_\gibbsmeas) =\hausdim(\attractor)$. If $\steplenlo{\steplenpot}{v}=\steplenup{\steplenpot}{v}$, then $\alpha_0=\steplenlo{\steplenpot}{v}=\steplenup{\steplenpot}{v}$. Otherwise, $\steplenlo{\steplenpot}{v}<\alpha_0<\steplenup{\steplenpot}{v}$.
    \end{cor}
        
    
    There have been many known results on the \holder regularity of Gibbs measures on the real line. Since a probability measure on the real line is completed determined by its cumulative distribution function, we shall state the known results in terms of Gibbs distribution functions rather than Gibbs measures. In what follows, let $\steplenpot:\shiftsp\rightarrow\bbR$ be a \holder continuous function satisfying that $P(\steplenpot)=0$. Let $\gibbsmeas$ be a Gibbs measure of $\steplenpot$. Until the end of the this introduction, we only consider the case where $\onedimmani=\bbR$. Hence, $\coding_*\gibbsmeas$ is a measure on $\bbR$. Let $\cdf:\bbR\rightarrow[0,+\infty)$ denote the cumulative distribution function of $\coding_*\gibbsmeas$.
    
    There are some articles studying the \holder differentiability of $\cdf$, but with a focus on sets different from what we considered in Theorem \ref{thm:app}. As an example, Kesseb\"ohmer and Stratmann evaluated in \cite{kessestratholder} the Hausdorff dimension of the set
    \begin{equation*}
        \set{x\in\bbR| \lodert^\alpha\cdf(x)<\updert^\alpha\cdf(x)=+\infty}.
    \end{equation*}
    At these points, $\cdf$ fails to be $\alpha$-\holder differentiable.
	
    There are also a series of articles studying the pointwise \holder exponent of the distribution function, which are closely related to our result. In \cite{patzschke}, Patzschke showed that the Hausdorff dimension of
    \begin{equation} \label{eqn:patzschkeset}
        \mathcal{E}^\alpha\defeq \Set{x\in\bbR |\lim_{y\rightarrow x}\frac{\log|\cdf(x)-\cdf(y)|}{\log|x-y|}=\alpha}
    \end{equation}
    is equal to $\udim(\levelset^\alpha_{\steplenpot,-\geopot})$, where $\geopot$ is the geometric potential defined in the same way as in Theorem \ref{thm:app}. As a continuation of this result, Jaerisch and Sumi studied extensively in \cite{JaerischSumi} various types of sets that contain $\mathcal{E}^\alpha$ as a subset, including
    \begin{align*}
        &\Set{x\in\bbR| \liminf_{y\rightarrow x}\frac{\log|\cdf(x)-\cdf(y)|}{\log|x-y|}=\alpha};\\
        &\Set{x\in\bbR| \limsup_{y\rightarrow x}\frac{\log|\cdf(x)-\cdf(y)|}{\log|x-y|}=\alpha};\\
        &\Set{x\in\bbR|\exists (y_n)_{n=1}^\infty\text{ converging to }x\text{ such that } \lim_{n\rightarrow+\infty}\frac{\log|\cdf(x)-\cdf(y_n)|}{\log|x-y_n|}=\alpha}.
    \end{align*}
    They proved that all these sets have the same Hausdorff dimension as $\mathcal{E}^\alpha$. The Hausdorff dimension and also the packing dimension of some sets even larger than the sets in \cite{JaerischSumi} are given in \cite{bos}. Different from \cite{bos} and \cite{JaerischSumi}, in which the sets larger than $\mathcal{E}^\alpha$ are considered, the set $\moderate^\alpha_\gibbsmeas$ we considered in Theorem \ref{thm:app} is actually a subset of $\mathcal{E}^\alpha$. Hence, our result is a continuation of the previous works on the pointwise \holder exponent.

    It is also worth noticing that the results in \cite{JaerischSumi} give the Hausdorff dimension of
    \begin{equation*}
        \mathcal{H}^\alpha\defeq \Set{x\in\bbR| \sup\Set{\gamma>0| \updert^{\gamma}\cdf(x)<+\infty}=\alpha}.
    \end{equation*}
    This set is interesting in that at points in $\mathcal{H}^\alpha$, $\alpha$ can be regarded as the exact \holder exponent locally. Indeed, the Hausdorff dimension of $\mathcal{H}^\alpha$ is also equal to the Hausdorff dimension of $\mathcal{E}^\alpha$ \cite{JaerischSumi}. Our Theorem \ref{thm:app} shows that
    \begin{equation*}
        \Set{x\in\bbR|0< \lodert^\alpha\cdf(x)\leq\updert^\alpha\cdf(x)<+\infty},
    \end{equation*}
    which is a subset of $\mathcal{H}^\alpha$, has the same Hausdorff dimension as $\mathcal{H}^\alpha$. Therefore, our result also complements \cite{JaerischSumi} in this sense.


    The rest of this paper is organized as follows. Section \ref{sec:known} serves as a summary of the known results we shall use. Section \ref{sec:dimspecunif} gives all the proofs of our results stated for symbolic dynamics, including the proofs of the two main theorems, namely Theorem \ref{thm:main:freq} and Theorem \ref{thm:main:notfreq}. Besides, remarks on the boundary case will be made at the end of Section \ref{sec:dimspecunif}. In Section \ref{sec:app}, we perform a multifractal analysis of \holder regularity for Gibbs measures on one-dimensional manifolds. The details of the setup and the proof of Theorem \ref{thm:app} will be provided in Section \ref{sec:app}. Appendix~\ref{app:key} gives the proofs of two lemmas we claimed in Section \ref{sec:dimspecunif}.
    
    \emph{Acknowledgements.} The author would like to thank Johannes Jaerisch for his invaluable comments and suggestions on this paper.

\section{Preliminaries}\label{sec:known}
    This section lays the background for our subsequent discussions, especially for the proofs in Section \ref{sec:dimspecunif}.
    
    \subsection{Notations and Basic Facts about SFT}\label{subsec:pre:sft}
        An SFT is specified by a finite set $\symbset$ of symbols and an incidence matrix $\incmat:\symbset\times \symbset\rightarrow\set{0,1}$ satisfying that for each $a\in\symbset$, there exists some $b\in\symbset$ such that $\incmat(a,b)=1$. We shall always assume that $\symbset$ contains at least two symbols. A non-empty word (of finite length) or a sequence (of infinite length) over $\symbset$ is called admissible if any of its subwords of length two, say $ab$, satisfies that $\incmat(a,b)=1$. It is vacuously true that the empty word as well as all words of length one are admissible. Then the shift space $\shiftsp\subseteq\symbset^\bbN$ is the set of all admissible sequences over $\symbset$. Here, $\bbN$ denotes the set of positive integers. The (one-sided) shift $\lshift:\shiftsp\rightarrow\shiftsp$ is defined by $\lshift(\xi_1\xi_2\cdots)\defeq\xi_2\xi_3\cdots$, for any $\xi=\xi_1\xi_2\cdots \in\shiftsp$. The set of all the admissible words of length $n$ will be denoted by $\symbset^n_\incmat$. Let $\symbset^*_\incmat\defeq\bigcup_{n=0}^\infty\symbset^n_\incmat$ be the set of all admissible words.
			
        The topology of $\shiftsp$ is generated by the cylinder sets. For any $\omega\in\symbset^*_\incmat$, define the cylinder set of $\omega$ as $[\omega]\defeq\set{\xi\in\shiftsp|\xi_1\cdots\xi_{|\omega|}=\omega}$, where $|\omega|$ denotes the length of $\omega$. If $\omega$ is the empty word, then $[\omega]=\shiftsp$. An $n$-cylinder set is a cylinder set given by a word of length $n$. The collection of all the cylinder sets forms a basis that generates a topology. Endowed with this topology, $\shiftsp$ is compact and metrizable. In the next subsection, we shall define a family of metrics compatible with this topology.
			
    \subsection{Metrics on Shift Space}\label{subsec:metricshiftsp}
        In this subsection, we define a family of metrics. Throughout this paper, the metric on $\shiftsp$ will always be picked from this family.
		
        One standard choice of a metric on $\shiftsp$ is $\symbdist_1$, defined by $\symbdist_1(\xi,\xi')\defeq\exp(-|\xi\wedge\xi'|)$ for any two distinct $\xi,\xi'\in\shiftsp$, where $\xi\wedge\xi'$ denotes the longest common initial block shared by $\xi$ and $\xi'$. Clearly, $d_1$ induces the topology generated by the cylinder sets.
	
        For any two metric spaces $(X,d_X)$ and $(Y,d_Y)$, a map $f:X\rightarrow Y$ is said to be \holder continuous if there exists some \holder exponent $\alpha\in\bbR$ such that
        \begin{equation*}
            \sup \Set{\frac{d_Y(f(x),f(x'))}{d_X(x,x')^\alpha}|x,x'\in X,\,x\neq x'}<+\infty.
        \end{equation*}
        Let $f:\shiftsp\rightarrow\bbR$ be a \holder continuous function, where $\shiftsp$ and $\bbR$ are endowed with $\symbdist_1$ and the Euclidean metric respectively. A prominent feature of a \holder continuous function $f$ is the bounded distortion property, which states that the distortion constant
        \begin{equation*}
            \distconst_f\defeq \sup_{\omega\in\symbset^*_\incmat}\sup_{\xi,\xi'\in[\omega]}\left|S_{\left|\omega\right|}f(\xi)-S_{\left|\omega\right|}f(\xi')\right|
		\end{equation*}
        is finite. For any admissible word $\omega$, we define
        \begin{equation*}
            S_\omega f\defeq\sup\set{S_{\left|\omega\right|}f(\xi)|\xi\in[\omega]}.
        \end{equation*}
        Thus, by the bounded distortion property of the \holder continuous function $f$, for every $\xi\in[\omega]$, we have $S_\omega f-\distconst_f\leq S_{|\omega|}f(\xi)\leq S_\omega f$.
	
        Now we associate a metric $d_\geopot$ to every \holder continuous $\geopot:\shiftsp\rightarrow(0,+\infty)$, where the metric of $\shiftsp$ is still taken to be $d_1$. Define a metric $\symbdist_\geopot$ by setting
        \begin{equation*}
            \symbdist_\geopot (\xi,\xi')\defeq \exp(S_{\xi\wedge\xi'}(-\geopot)),
        \end{equation*}
        for any two distinct $\xi,\xi'\in\shiftsp$. When $\geopot$ is constantly equal to $1$, we can readily see that $\symbdist_\geopot$ is precisely $\symbdist_1$, so this definition is consistent with the previously introduced notation $\symbdist_1$.

        The \holder continuity of any function $f$ does not depend on which metric we pick from the family we have just defined.
        \begin{prop} \label{prop:holdequiv}
            Let $\geopot:\shiftsp\rightarrow(0,+\infty)$ be \holder continuous with respect to $d_1$. Then, for any $f:\shiftsp \rightarrow\bbR$, $f$ is \holder continuous with respect to $d_1$ if and only if it is \holder continuous with respect to $d_\geopot$.
        \end{prop}
        \begin{proof}
            Since $\geopot$ is a positive continuous function on the compact space $\shiftsp$, we have
            \begin{equation*}
                0< \min_{\zeta \in\shiftsp}\geopot(\zeta) \leq\max_{\zeta \in\shiftsp}\geopot(\zeta)<+\infty.
            \end{equation*}
            Also note that, for any $\xi,\xi'\in\shiftsp$,
            \begin{equation*}
                \min_{\zeta \in\shiftsp} \geopot(\zeta)\log\symbdist_1(\xi,\xi')\leq\log\symbdist_\geopot(\xi,\xi')\leq\max_{\zeta\in\shiftsp}\geopot(\zeta)\log\symbdist_1(\xi,\xi').
            \end{equation*}
            From these facts and the definition of \holder continuity, our claim follows.
        \end{proof}
        
        Recall that we have fixed a \holder continuous function $\geopot:\shiftsp\rightarrow(0,+\infty)$ with respect to $d_1$. The metric $\symbdist_\geopot$ assigns to every subset of $\shiftsp$ a non-negative number as its Hausdorff dimension. In \cite{barsch}, the authors introduced the notion of $\geopot$-dimension as an extension of the notion of the Hausdorff dimension. In our setting, we always assume the \holder continuity of $\geopot$ with respect to $d_1$. Then, the $\geopot$-dimension $\udim(E)$ of a set $E\subseteq\shiftsp$ coincides with its Hausdorff dimension $\hausdim(E)$, if we take the metric of $\shiftsp$ to be $d_\geopot$. In what follows, we shall study the $\geopot$-dimension of subsets of $\shiftsp$. Thus, it is, on the one hand, convenient for us to take $d_\geopot$ as the default metric on $\shiftsp$. On the other hand, a majority of the authors working in thermodynamic formalism usually assume \holder continuity with respect to $\symbdist_1$. By Proposition \ref{prop:holdequiv}, the different choices of the metrics in fact do not make any difference when we speak of the \holder continuity of a function.
        
        Lastly, we assume that the SFT we consider is always topologically transitive. It is broadly known that the SFT is topologically transitive if and only if there exists a finite set $\infixset$ of admissible words such that for any two symbols $a,b\in\symbset$, we can find some $\rho\in\infixset$ such that $a\rho b$ is admissible.

        For a topological transitive SFT $\lshift:\shiftsp\rightarrow\shiftsp$, we have the following classical results from thermodynamic formalism. Most of the textbooks on thermodynamic formalism contain these results; for instance, see \cite{pu}. Take an arbitrary \holder continuous function $f:\shiftsp\rightarrow\bbR$. The topological pressure of $f$ is
        \begin{equation*}
            P(f)\defeq \lim_{n\rightarrow+\infty}\frac{1}{n}\log\sum_{|\omega|=n}\exp(S_\omega f).
        \end{equation*}
        A Gibbs measure of potential $f$ is a Borel probability measure $\gibbsmeas$ admitting the existence of some $C_\gibbsmeas\geq 1$ such that for any non-empty admissible word $\omega$,
        \begin{equation*}
            C_\gibbsmeas^{-1} \exp(S_\omega f-|\omega|\cdot P(f)) \leq\gibbsmeas([\omega])\leq C_\gibbsmeas\exp(S_\omega f-|\omega|\cdot P(f)).
        \end{equation*}
        There is a standard argument using the Perron-Frobenius operators to show the existence of the $\lshift$-invariant Gibbs measure $\gibbsmeas$ whose potential is $f$. This $\lshift$-invariant Gibbs measure $\gibbsmeas$ is known to be ergodic. Moreover, $\gibbsmeas$ is also the unique equilibrium state of $f$, meaning that $P(f)=h(\gibbsmeas)+ \int_\shiftsp f\,\rmd\gibbsmeas$, where
        \begin{equation*}
            h(\gibbsmeas)\defeq \lim_{n\rightarrow+\infty}-\frac{1}{n}\sum_{|\omega|=n}\gibbsmeas([\omega])\log([\omega])
        \end{equation*}
        is the Kolmogorov-Sinai entropy of the measure $\gibbsmeas$.
    \subsection{Multifractal Analysis of Gibbs Measures} \label{subsec:pre:gibbs}
        In this subsection, we recall some widely-known facts established in the multifractal analysis of Gibbs measures. As in (\ref{eqn:steplenloupdef}), for $\steplenpot:\shiftsp\rightarrow\bbR$ and $v:\shiftsp\rightarrow(-\infty,0)$, we define
        \begin{equation*}
            \steplenlo{\steplenpot}{v}\defeq\inf_{\xi\in\shiftsp}\liminf_{n\rightarrow+\infty}\frac{S_n\steplenpot(\xi)}{S_nv(\xi)};\quad\steplenup{\steplenpot}{v}\defeq\sup_{\xi\in\shiftsp}\limsup_{n\rightarrow+\infty}\frac{S_n\steplenpot(\xi)}{S_nv(\xi)}.
        \end{equation*}
        The $\geopot$-dimension of a Borel measure $\mu$ on $\shiftsp$ is
        \begin{equation*}
            \udim(\mu)= \inf\set{\udim(E)|E\subseteq \shiftsp,\,\mu(\shiftsp\setminus E)=0}.
        \end{equation*}
        A thorough treatment of the dimension of a measure as well as its relation with the pointwise dimension of the measure can be found in \cite[Chapter 8]{pu}. In this paper, we shall omit the details.

        What we need from the multifractal analysis is summarized into the following theorem.
        \begin{thm}[{\cite{schmeling}}]\label{thm:pworiginal}
            Let $\lshift:\shiftsp\rightarrow\shiftsp$ be a topologically transitive SFT. Let $\steplenpot:\shiftsp\rightarrow\bbR$ and $\geopot:\shiftsp\rightarrow(0,+\infty)$ be \holder continuous functions. Then, $\levelset^\alpha_{\steplenpot,-\geopot}$ is non-empty if and only if $\steplenlo{\steplenpot}{-\geopot}\leq\alpha\leq\steplenup{\steplenpot}{-\geopot}$.
            
            For each $q\in\bbR$, there exists a unique $\fenerg(q)$ such that $P(q\steplenpot-\fenerg(q)\geopot)=0$. This defines a function $\fenerg:\bbR\rightarrow\bbR$. The function $\fenerg$ is real analytic and convex.
            
            The function $\fenerg$ gives an expression of $\udim(\levelset^\alpha_{\steplenpot,-\geopot})$ in the following manner. If we assume that $\steplenlo{\steplenpot}{-\geopot}<\steplenup{\steplenpot}{-\geopot}$, then the following statements hold.
            \begin{enumerate}
                \item   For every $\alpha\in(\steplenlo{\steplenpot}{-\geopot},\steplenup{\steplenpot}{-\geopot})$, there exists a unique $q_\alpha\in\bbR$ such that
                \begin{equation*}
                    \alpha=-\fenerg'(q_\alpha)=\frac{\int_\shiftsp\steplenpot\,\rmd\gibbsmeas_\alpha}{-\int_\shiftsp\geopot\,\rmd\gibbsmeas_\alpha},
                \end{equation*}
                where $\gibbsmeas_\alpha$ is the equilibrium state of the potential $q_\alpha\steplenpot-\fenerg(q_\alpha)\geopot$. We have $\gibbsmeas_\alpha(\levelset^\alpha_{\steplenpot,-\geopot})=1$, and the $\geopot$-dimension of $\levelset^\alpha_{\steplenpot,-\geopot}$ is given by
                \begin{equation*}
                    \udim( \levelset^\alpha_{\steplenpot,-\geopot})=\udim(\gibbsmeas_\alpha)=\fenerg(q_\alpha)+\alpha q_\alpha=\min_{q\in\bbR}\fenerg(q)+\alpha q.
                \end{equation*}
                \item   For every $\alpha\in\set{\steplenlo{\steplenpot}{-\geopot},\steplenup{\steplenpot}{-\geopot}}$, there exists a $\lshift$-invariant Borel probability measure $\gibbsmeas_\alpha$ such that $\gibbsmeas_\alpha(\levelset^\alpha_{\steplenpot,-\geopot})=1$ and
                \begin{equation*}
                    \udim( \levelset^\alpha_{\steplenpot,-\geopot})=\udim(\gibbsmeas_\alpha)=\inf_{q\in\bbR}\fenerg(q)+\alpha q.
                \end{equation*}
            \end{enumerate}
        \end{thm}

        \begin{rmk}
            It is assumed in \cite{schmeling} that $\lshift:\shiftsp\rightarrow\shiftsp$ is topologically mixing. In fact, Theorem \ref{thm:pworiginal} holds also for a topologically transitive SFT. Besides, it is also assumed in \cite{schmeling} that $P(\steplenpot)=0$. As all the claims in Theorem \ref{thm:pworiginal} still hold even if $\steplenpot$ does not have zero topological pressure, we also remove this assumption.
        \end{rmk}

        In the statement of Theorem \ref{thm:pworiginal}, the $\alpha$-level set is $\levelset^\alpha_{\steplenpot,-\geopot}$. Now we rewrite Theorem \ref{thm:pworiginal} to give dimension formulae and other statements about $\levelset^\alpha_{\steplenpot,v}$ as we considered in the introduction. Towards this end, for a given $\alpha$, we define $\fenerg_\alpha:\bbR\rightarrow\bbR$ in a way similar to how we defined $\fenerg$ in Theorem \ref{thm:pworiginal}.
        
        Let $\lshift:\shiftsp\rightarrow\shiftsp$ be a topologically transitive SFT. Let $\steplenpot:\shiftsp\rightarrow\bbR$, $v:\shiftsp\rightarrow\bbR$ and $\geopot:\shiftsp\rightarrow(0,+\infty)$ be \holder continuous functions. Assume that $v$ is either everywhere positive or everywhere negative. For every $q\in\bbR$, let $\fenerg_\alpha(q)$ be the unique real number for which $P(q(\steplenpot-\alpha v)-\fenerg_\alpha(q)\geopot)=0$. Evidently, if $v=-\geopot$, then we have $\fenerg_\alpha(q)=\fenerg(q)+\alpha q$ for any $q\in\bbR$, where $\fenerg$ is the function defined in Theorem \ref{thm:pworiginal}. For every $\alpha$, applying Theorem \ref{thm:pworiginal} to $\steplenpot-\alpha v$ in place of $\steplenpot$, then we can readily see that $\fenerg_\alpha$ is also real analytic and convex.

        Now we may rewrite Theorem \ref{thm:pworiginal} as follows.
        \begin{thm}\label{thm:pw}
            Let $\lshift:\shiftsp\rightarrow\shiftsp$ be a topologically transitive SFT. Let $\steplenpot:\shiftsp\rightarrow\bbR$, $v:\shiftsp\rightarrow(-\infty,0)$ and $\geopot:\shiftsp\rightarrow(0,+\infty)$ be \holder continuous functions. Then, $\levelset^\alpha_{\steplenpot,v}$ is non-empty if and only if $\steplenlo{\steplenpot}{v}\leq\alpha\leq\steplenup{\steplenpot}{v}$.
            
            Furthermore, if we assume that $\steplenlo{\steplenpot}{v}<\steplenup{\steplenpot}{v}$, then the following statements hold.
            \begin{enumerate}
                \item \label{item:pw:interior}   For any $\alpha\in(\steplenlo{\steplenpot}{v},\steplenup{\steplenpot}{v})$, there exists a unique $q_\alpha\in\bbR$ such that
                \begin{equation*}
                    \fenerg'_\alpha (q_\alpha) =\frac{\int_{\shiftsp}(\steplenpot-\alpha v)\,\rmd\gibbsmeas_\alpha}{\int_{\shiftsp}\geopot\,\rmd\gibbsmeas_\alpha}=0,
                \end{equation*}
                where $\gibbsmeas_\alpha$ is the equilibrium state of the potential $q_\alpha(\steplenpot-\alpha v)-\fenerg_\alpha(q_\alpha)\geopot$. We have $\gibbsmeas_\alpha(\levelset^\alpha_{\steplenpot,v})=1$ and
                \begin{equation*}
                    \udim( \levelset^\alpha_{\steplenpot,v})=\udim(\gibbsmeas_\alpha)=\min_{q\in\bbR}\fenerg_\alpha(q)=\fenerg_\alpha(q_\alpha).
                \end{equation*}
                \item \label{item:pw:bdy}   For any $\alpha\in\set{\steplenlo{\steplenpot}{v},\steplenup{\steplenpot}{v}}$, there is a $\lshift$-invariant Borel probability measure $\gibbsmeas_\alpha$ such that $\gibbsmeas_\alpha(\levelset^\alpha_{\steplenpot,v})=1$ and
                \begin{equation*}
                    \udim( \levelset^\alpha_{\steplenpot,v})=\udim( \gibbsmeas_\alpha)=\inf_{q\in\bbR}\fenerg_\alpha(q).
                \end{equation*}
            \end{enumerate}
        \end{thm}

        Theorem \ref{thm:pw} can be deduced from Theorem \ref{thm:pworiginal}. The critical observation is that $\levelset^\alpha_{\steplenpot,v}=\levelset^0_{\steplenpot-\alpha v,-\geopot}$ for every $\alpha$. Hence, the main idea of the proof is to replace $\steplenpot$ in Theorem \ref{thm:pworiginal} with $\steplenpot-\alpha v$. The complete proof is very tedious, so we omit it.
        
        The case where $\steplenlo{\steplenpot}{v}=\steplenup{\steplenpot}{v}$ is not covered by Theorem \ref{thm:pw}. Actually, this case will be well understood through the discussions in the next subsection. As we shall see in Corollary \ref{cor:bd}, if $\steplenlo{\steplenpot}{v}=\steplenup{\steplenpot}{v}$, then we have $\levelset^\alpha_{\steplenpot,v}=\shiftsp$ for $\alpha=\steplenlo{\steplenpot}{v}$, and $\levelset^\alpha_{\steplenpot,v}=\varnothing$ for $\alpha\neq\steplenlo{\steplenpot}{v}$.
    \subsection{Facts from Ergodic Optimization}
        When we deal with the boundary case, we will need some facts in ergodic optimization.
        \begin{thm}[\cite{ergopt}] \label{thm:subact}
            Let $\lshift:\shiftsp\rightarrow\shiftsp$ be a topologically transitive SFT. Let $\steplenpot:\shiftsp\rightarrow\bbR$ and $v:\shiftsp\rightarrow(-\infty,0)$ be \holder continuous. Then, 
            \begin{enumerate}
                \item   for $\alpha=\steplenlo{\steplenpot}{v}$, there exists a \holder continuous function $g_-:\shiftsp\rightarrow\bbR$ such that $\steplenpot-\alpha v+g_-\circ\lshift-g_-\leq 0$;
                \item   for $\alpha=\steplenup{\steplenpot}{v}$, there exists a \holder continuous function $g_+:\shiftsp\rightarrow\bbR$ such that $\steplenpot-\alpha v+g_+\circ\lshift-g_+\geq 0$.
            \end{enumerate}
        \end{thm}
        The functions $g_-$ and $g_+$ above are called sub-actions in the context of ergodic optimization. The existence of sub-actions relies on the \holder continuity of $\steplenpot$ \cite{ergopt}, which we always assume throughout this paper.
        
        Using the existence of sub-actions, we can show the following facts.
        \begin{cor}\label{cor:bd}
            Let $\lshift:\shiftsp\rightarrow\shiftsp$ be a topologically transitive SFT. Let $\steplenpot:\shiftsp\rightarrow\bbR$ and $v:\shiftsp\rightarrow(-\infty,0)$ be \holder continuous functions. Then,
            \begin{enumerate}
                \item   $\sup_{\xi\in\shiftsp}\sup_{n\geq 1}S_n\steplenpot(\xi)-\alpha S_nv(\xi)<+\infty$ when $\alpha=\steplenlo{\steplenpot}{v}$;
                \item   $\inf_{\xi\in\shiftsp}\inf_{n\geq 1}S_n\steplenpot(\xi)-\alpha S_nv(\xi)>-\infty$ when $\alpha=\steplenup{\steplenpot}{v}$;
                \item   $\uniflevset^\alpha_{\steplenpot,v}=\levelset^\alpha_{\steplenpot,v}=\shiftsp$ if $\steplenlo{\steplenpot}{v}=\steplenup{\steplenpot}{v}=\alpha$.
            \end{enumerate}
        \end{cor}
        \begin{proof}
            Suppose that $\alpha=\steplenlo{\steplenpot}{v}$. By Theorem \ref{thm:subact}, there is a sub-action $g_-:\shiftsp\rightarrow\bbR$ such that $\steplenpot-\alpha v+g_-\circ\lshift-g_-\leq 0$. Then, we have that for any $\xi\in\shiftsp$ and any positive integer $n$,
            \begin{equation*}
                S_n\steplenpot(\xi)-\alpha S_nv(\xi)= S_n(\steplenpot-\alpha v+g_-\circ\lshift-g_-)(\xi)+g_-(\lshift^n(\xi))-g_-(\xi)\leq 2\|g_-\|,
            \end{equation*}
            where $\|g_-\|\defeq\sup_{\xi\in\shiftsp}|g_-(\xi)|$ is the supremum norm of the continuous function $g_-$. This shows the first claim. Applying the first claim to $-\steplenpot$ and $v$, we have the second claim. Combining the first two assertions, we have $\uniflevset^\alpha_{\steplenpot,v}=\shiftsp$ when $\steplenlo{\steplenpot}{v}=\steplenup{\steplenpot}{v}=\alpha$. Since $\uniflevset^\alpha_{\steplenpot,v}\subseteq\levelset^\alpha_{\steplenpot,v}\subseteq\shiftsp$, the last claim holds.
        \end{proof}
\section{Proofs of Main Theorems and Related Remarks}\label{sec:dimspecunif}
    In this section, we fix a topological transitive SFT $\lshift:\shiftsp\rightarrow\shiftsp$ and three \holder continuous functions $\steplenpot:\shiftsp\rightarrow \bbR$, $v:\shiftsp\rightarrow(-\infty,0)$ and $\geopot:\shiftsp\rightarrow(0,+\infty)$. For every $\alpha\in\bbR$, we write $\steplenpot_\alpha$ to denote $\steplenpot-\alpha v$.
    
    Define $\|\cylchain\|\defeq\sup_{\omega\in\cylchain}|\omega|$ for any $\cylchain \subseteq\symbset^*_\incmat$. In addition, for any $K>0$, any $\alpha\in\bbR$ and any positive integer $m$, define
    \begin{equation*}
        \windowword^m_{\alpha,K} \defeq\Set{\omega\in\symbset^m_\incmat|\sup_{\xi\in[\omega]}\left|S_m\steplenpot_\alpha(\xi)\right|\leq K},
    \end{equation*}
    and $\windowword_{\alpha,K}\defeq\bigcup_{m=0}^\infty\windowword^m_{\alpha,K}$.
    \subsection{Proof of Theorem \ref{thm:main:freq}}
        We first prove Theorem \ref{thm:main:freq}. Towards this end, we state three lemmas. The proofs of the first two will be provided in Appendix \ref{app:key}. The last lemma is a known fact, so the proof will be omitted.

        \begin{lma} \label{lma:infixset}
            Suppose that $\udim(\shiftsp)>0$. Then, there exists a finite set $\infixset$ of non-empty admissible words satisfying that
            \begin{enumerate}
                \item   for any two symbols $a,b\in\symbset$, there exists some $\rho\in\infixset$ such that $a\rho b$ is admissible;
                \item   for any two distinct words $\rho,\rho'\in\infixset$, the cylinder set $[\rho]$ is disjoint from $[\rho']$.
            \end{enumerate}
        \end{lma}
        This lemma is trivial if we assume that $\lshift:\shiftsp\rightarrow\shiftsp$ is topologically mixing, for we can take all the admissible words in $\infixset$ to have the same length. Since we assumed that $\lshift:\shiftsp\rightarrow\shiftsp$ is merely topologically transitive, we need Lemma \ref{lma:infixset}. Henceforth, we shall fix a set $\infixset$ satisfying the conditions stated above.

        Lemma \ref{lma:gjk} below plays a role similar to Lemma 3.8 in \cite{GJK}. The proof we will give in Appendix \ref{app:key} is more elementary than the one in \cite{GJK}.
        \begin{lma}\label{lma:gjk}
            Suppose that $\steplenlo{\steplenpot}{v}<\alpha<\steplenup{\steplenpot}{v}$. Then, given any two constants $K>2\distconst_{\steplenpot_\alpha}+\|\infixset\|\cdot\|\steplenpot_\alpha\|$ and $K'>0$, there exists a finite family $\postfixset$ of admissible words such that for any $\omega\in\windowword_{\alpha,K'}$, there is some $\tau\in\postfixset$ for which $\omega\tau\in\windowword_{\alpha,K}$.
        \end{lma}
        
        We also need the following known result on the recurrence of a random walk on $\bbR$.
        \begin{lma}[{\cite{atkinson,GJK}}]\label{lma:atkinson}
            Given any ergodic $\lshift$-invariant Borel probability measure $\mu$ on $\shiftsp$ and any $\mu$-integrable function $f:\shiftsp\rightarrow\bbR$, we have that $\int_{\shiftsp}f\,\rmd\mu=0$ if and only if $\liminf\limits_{n\rightarrow+\infty}|S_nf|=0$ $\mu$-a.e.
        \end{lma}

        Now we give the proof of Theorem \ref{thm:main:freq}, which is greatly inspired by \cite{GJK}.
        \begin{proof}[Proof of Theorem \ref{thm:main:freq}]
            When $\alpha\notin[\steplenlo{\steplenpot}{v},\steplenup{\steplenpot}{v}]$, we have $\unifrecset_\steplenpot\cap\freqappearset_\wordset=\recset_\steplenpot=\varnothing$, so our claim holds trivially. Thus, we only need to handle the case where $\steplenlo{\steplenpot}{v}<\alpha<\steplenup{\steplenpot}{v}$. Assume that $\steplenlo{\steplenpot}{v}<\alpha <\steplenup{\steplenpot}{v}$ for the rest of the proof. In addition, without loss of generality, we assume that $\udim(\shiftsp)>0$.
            
            As the equilibrium state $\gibbsmeas_\alpha$ is ergodic and $\int_{\shiftsp}\steplenpot_\alpha\,\rmd\gibbsmeas_\alpha=0$, by Lemma \ref{lma:atkinson}, we have $\liminf\limits_{n\rightarrow+\infty}\left|S_n\steplenpot_\alpha\right|=0$, $\gibbsmeas_\alpha$-a.e. Hence, by Borel-Cantelli lemma, for any $K>2\distconst_{\steplenpot_\alpha}+\|\infixset\|\cdot\|\steplenpot\|\geq\distconst_{\steplenpot_\alpha}$,
            \begin{equation*}
                \sum_{\omega\in \windowword_{\alpha,K}} \gibbsmeas_\alpha ([\omega])=+\infty
            \end{equation*}
            Combining the divergence of this series with the Gibbs property of $\gibbsmeas_\alpha$, we have
            \begin{equation*}
                \sum_{\omega\in \windowword_{\alpha,K}}\exp(S_\omega (q_\alpha\steplenpot_\alpha-\fenerg_\alpha(q_\alpha)\geopot))=+\infty.
            \end{equation*}
            This in turn gives
            \begin{equation*}
                \sum_{\omega \in\windowword_{\alpha,K}}\exp(-\fenerg_\alpha(q_\alpha)S_\omega\geopot)\geq\exp(-Kq_\alpha)\sum_{\omega\in\windowword_{\alpha, K}}\exp(S_\omega (q_\alpha\steplenpot_\alpha-\fenerg_\alpha(q_\alpha)\geopot))=+\infty,
            \end{equation*}
            where $q_\alpha$ is, as in Theorem \ref{thm:pw}, the unique real number satisfying that $\fenerg_\alpha(q_\alpha)=\udim(\levelset^\alpha_{\steplenpot,v})$. For any positive $s<\fenerg_\alpha(q_\alpha)=\udim(\levelset^\alpha_{\steplenpot,v})$, pick an integer $m\geq 1$ such that
            \begin{equation} \label{eqn:poinseries}
                \frac{1}{s} \log\sum_{\omega\in\windowword^m_{\alpha,K}}\exp(-sS_\omega\geopot)>\left(2\|\infixset\|+\|\postfixset\|+|\tilde{\omega}|\right)\|\geopot\|.
            \end{equation}
            To ease the notation, we shall write $C_0$ to denote $\left(2\|\infixset\|+\|\postfixset\|+|\tilde{\omega}|\right)\|\geopot\|$.
            
            Let $\tilde{\omega}$ be an admissible word in which all words from $\wordset$ appear at least once. Let $\postfixset$ be a finite set of admissible words satisfying the condition in Lemma \ref{lma:gjk}, with $K'$ taken to be $K+\left(2\|\infixset\|+|\tilde{\omega}|\right)\|\steplenpot\|$.
            
            For any admissible word $\omega$, $\freqappearset_\omega$ will denote $\freqappearset_{\set{\omega}}$. We claim that $\udim(\uniflevset^\alpha_{\steplenpot,v}\cap\freqappearset_{\tilde{\omega}})\geq s$. To show this claim, we construct a measure $\mu_s$ for which $\mu_s(\uniflevset^\alpha_{\steplenpot,v}\cap\freqappearset_{\tilde{\omega}})=1$, and then apply the mass distribution principle to this $\mu_s$. Define by induction a sequence $(\newsymbset_n)_{n\geq1}$ of subsets of $\windowword_{\alpha,K}$ as follows. Let $\newsymbset_1\defeq\windowword^m_{\alpha,K}$. For $k\geq 2$, fix $\omega^{(k-1)}\in\newsymbset_{k-1}$ and $\omega\in\newsymbset_1$, and choose $\rho,\lambda\in\infixset$ satisfying that $\omega^{(k-1)}\rho\omega\lambda\tilde{\omega}$ is admissible. Then, since $\newsymbset_{k-1}$ and $\newsymbset_1$ are, by the induction hypothesis, both subsets of $\windowword_{\alpha,K}$, we have
            \begin{equation*}
                \sup_{\xi\in [\omega^{(k-1)}\rho\omega\lambda\tilde{\omega}]}\left|S_{|\omega^{(k-1)}\rho\omega\lambda\tilde{\omega}|}\steplenpot_\alpha(\xi)\right|\leq K'.
            \end{equation*}
            Therefore, by Lemma \ref{lma:gjk}, there exists some $\tau\in\postfixset$ such that $\omega^{(k-1)}\rho\omega\lambda\tilde{\omega}\tau\in\windowword_{\alpha,K}$. Note that the word $\omega^{(k-1)}\rho\omega\lambda\tilde{\omega}\tau$ is constructed from $\omega^{(k-1)}\in\newsymbset_{k-1}$ and $\omega\in\newsymbset_1$, so we may denote it by $\theta_k(\omega^{(k-1)},\omega)$. For every $\omega^{(k-1)}\in\newsymbset_{k-1}$, define $\newsymbset_k(\omega)\defeq\set{\theta_k(\omega,\omega')|\omega'\in\newsymbset_1}$ and $\newsymbset_k\defeq\bigcup_{\omega\in\newsymbset_{k-1}}\newsymbset_k(\omega)$. It is clear from our discussion above that $\newsymbset_k\subseteq\windowword_{\alpha,K}$.
            
            Now we are ready to construct the mass distribution $\mu_s$ from the set sequence $(\newsymbset_n)_{n\geq1}$. Set
            \begin{equation*}
                \mu_s([\omega^{(1)}]) \defeq\frac{\exp(-sS_{\omega^{(1)}}\geopot)}{\sum_{\omega\in\newsymbset_1}\exp(-sS_{\omega}\geopot)}
            \end{equation*}
            for any $\omega^{(1)}\in\newsymbset_1$ and
            \begin{equation*}
                \mu_s([\omega^{(k)}]) \defeq\frac{\exp(-sS_{\omega^{(k)}}\geopot)\mu_s([\omega^{(k-1)}])}{\sum_{\omega\in\newsymbset_k(\omega^{(k-1)})}\exp(-sS_{\omega}\geopot)}
            \end{equation*}
            for $\omega^{(k)}\in\newsymbset_k(\omega)$, where $k\geq 2$ and $\omega^{(k-1)}\in\newsymbset_{k-1}$. Clearly, for every $\omega^{(k-1)}\in\newsymbset_{k-1}$ and every $\omega^{(k)}\in\newsymbset_{k}(\omega)$, $\omega^{(k)}$ is a continuation of $\omega^{(k-1)}$, or equivalently, we can say $[\omega^{(k)}]\subseteq[\omega^{(k-1)}]$. Moreover, for any $\omega^{(k-1)}\in\newsymbset_{k-1}$, and any two distinct $\omega,\omega'\in\newsymbset_1$, we claim that
            \begin{equation} \label{eqn:thm:main:freq:proof}
                [\theta_k(\omega^{(k-1)},\omega)]\cap[\theta_k(\omega^{(k-1)},\omega')]=\varnothing.
            \end{equation}
            Once we manage to prove (\ref{eqn:thm:main:freq:proof}), we can immediately deduce the existence and uniqueness of the measure $\mu_s$ from Kolmogorov consistency theorem. The proof of (\ref{eqn:thm:main:freq:proof}) is given as follows. As before, we write $\theta_k(\omega^{(k-1)},\omega)=\omega^{(k-1)}\rho\omega\lambda\tilde{\omega}\tau$. For $\omega'$, we write $\theta_k(\omega^{(k-1)},\omega')=\omega^{(k-1)}\rho'\omega'\lambda'\tilde{\omega}\tau'$ in the same way as we wrote $\theta_k(\omega^{k-1},\omega)$. This means that $\rho$ and $\rho'$ are both words from $\infixset$. If $\rho\neq\rho'$, then we have
            \begin{equation*}
                [\theta_k(\omega^{(k-1)},\omega)]\cap[\theta_k(\omega^{(k-1)},\omega')]\subseteq[\omega^{(k-1)}\rho]\cap[\omega^{k-1}\rho']=\varnothing,
            \end{equation*}
            because the set $\infixset$ we took fulfills the conditions in Lemma \ref{lma:infixset}. If $\rho=\rho'$, then we have
            \begin{equation*}
                [\theta_k(\omega^{(k-1)},\omega)]\cap[\theta_k(\omega^{(k-1)},\omega')]\subseteq[\omega^{(k-1)}\rho\omega]\cap[\omega^{k-1}\rho'\omega']=\varnothing,
            \end{equation*}
            because $\omega$ and $\omega'$ are distinct words with the same length $m$, which in particular implies that $\omega$ is not an initial block of $\omega'$ and vice versa. Therefore, in any of the two possible cases above, (\ref{eqn:thm:main:freq:proof}) always holds.
            
            By construction, $\mu_s$ is supported on $\bigcap_{k=1}^\infty\bigcup_{\omega\in\newsymbset_k}[\omega]$. Also note that for any $\xi\in\bigcap_{k=1}^\infty\bigcup_{\omega\in\newsymbset_k}[\omega]$,
            \begin{equation*}
                \sup_{n\geq 1}|S_n\steplenpot_\alpha(\xi)|\leq K'+\|\postfixset\|\cdot\|\steplenpot_\alpha\|.
            \end{equation*}
            Moreover, all subwords of $\xi$ with length no less than $2(m+\|\infixset\|+|\tilde{\omega}|)+\|\postfixset\|$ contain $\tilde{\omega}$ as a subword, thus containing all the words from $\wordset$ as subwords. Therefore, we have $\mu_s(\uniflevset^\alpha_{\steplenpot,v}\cap\freqappearset_\wordset)=1$.
            
            Observe that for any integer $k\geq2$ and $\omega\in\newsymbset_{k-1}$,
            \begin{align*}
                \sum_{\omega'\in \newsymbset_k(\omega)}\exp(-sS_{\omega'}\geopot)&\geq\exp(-sS_\omega\geopot-sC_0)\sum_{\omega''\in\newsymbset_1=\windowword^m_K}\exp(-sS_{\omega''}\geopot)\\
                &\geq\exp(-sS_\omega\geopot).
            \end{align*}
            As a result, for any integer $k\geq 2$, 
            \begin{align*}
                \max_{\omega\in \newsymbset_k}\frac{\mu_s([\omega])}{\exp(-sS_\omega\geopot)}&=\max_{\omega'\in\newsymbset_{k-1}}\frac{\mu_s([\omega'])}{\sum_{\omega''\in\newsymbset_k(\omega')}\exp(-sS_{\omega''}\geopot)}\leq\max_{\omega'\in\newsymbset_{k-1}}\frac{\mu_s([\omega'])}{\exp(-sS_{\omega'}\geopot)},
            \end{align*}
            thus implying that
            \begin{equation*}
                \max_{\omega\in \newsymbset_k}\frac{\mu_s([\omega])}{\exp(-sS_\omega\geopot)}\leq\max_{\omega\in\newsymbset_1}\frac{\mu_s([\omega])}{\exp(-sS_{\omega}\geopot)}.
            \end{equation*}
            By the mass distribution principle, $\udim(\uniflevset^\alpha_{\steplenpot,v}\cap\freqappearset_\wordset)\geq s$. Since $s$ is an arbitrary positive number smaller than $\udim(\levelset^\alpha_{\steplenpot,v})$, we conclude that $\udim(\uniflevset^\alpha_{\steplenpot,v}\cap\freqappearset_\wordset)\geq\udim(\levelset^\alpha_{\steplenpot,v})$. Since $\uniflevset^\alpha_{\steplenpot,v}\cap\freqappearset_\wordset\subseteq\levelset^\alpha_{\steplenpot,v}$, the proof is complete.
        \end{proof}
    \subsection{Proof of Theorem \ref{thm:main:notfreq}}
        We employ Theorem \ref{thm:main:freq} to prove Theorem \ref{thm:main:notfreq}.
        \begin{proof}[Proof of Theorem \ref{thm:main:notfreq}]
            We construct an admissible word $\omega^*$ as follows. Let $\wordset^\infty\defeq\set{\omega\omega\cdots|\omega\in\wordset}\subseteq\shiftsp$. Note that $\bigcup_{n=0}^\infty\lshift^n\wordset^\infty$ is finite. Thus, we can take one admissible word $\omega^*$ whose cylinder set $[\omega^*]$ is disjoint from $\bigcup_{n=0}^\infty\lshift^n\wordset^\infty$. For this $\omega^*$, we claim that
            \begin{equation} \label{eqn:thm:main:notfreq:proof}
                \freqappearset_{ \omega^*} \defeq\freqappearset_{\set{\omega^*}}\subseteq\notfreq_\wordset.
            \end{equation}
            To see this, take an arbitrary $\xi\in\freqappearset_{\omega^*}$. Then, there exists some $l\in\bbN$ such that every subword of $\xi$ with length no less than $l$ has $\omega^*$ appearing therein. For any $\omega\in\wordset$, note that the word $\omega^l$ does not have $\omega^*$ as a subword, and has its length being $l\cdot|\omega|\geq l$. Hence, $\omega^l$ cannot appear in $\xi$, so $\xi\in\notfreq_\wordset$. This proves (\ref{eqn:thm:main:notfreq:proof}).
            
            Combining (\ref{eqn:thm:main:notfreq:proof}) with Theorem \ref{thm:main:freq}, we immediately have
            \begin{equation*}
                \udim( \uniflevset^\alpha_{\steplenpot,v}\cap\notfreq_\wordset)\geq\udim(\uniflevset^\alpha_{\steplenpot,v}\cap\freqappearset_{\omega^*})=\udim(\levelset^\alpha_{\steplenpot,v}),
            \end{equation*}
            for any $\alpha\notin\set{\steplenlo{\steplenpot}{v},\steplenup{\steplenpot}{v}}$. Since $\uniflevset^\alpha_{\steplenpot,v}\cap\notfreq_\wordset\subseteq \levelset^\alpha_{\steplenpot,v}$, the proof is complete.
        \end{proof}
    \subsection{Proof of Corollary \ref{cor:freqnotfreqmaxdim}} \label{subsec:cor:freqnotfreqmaxdim}
        Using Theorem \ref{thm:main:freq} and Theorem \ref{thm:main:notfreq}, we can prove Corollary \ref{cor:freqnotfreqmaxdim} as follows.
        \begin{proof}[Proof of Corollary \ref{cor:freqnotfreqmaxdim}]
            Take \holder continuous $\steplenpot:\shiftsp\rightarrow\bbR$ and $v:\shiftsp\rightarrow(-\infty,0)$ such that $\steplenlo{\steplenpot}{v}<\steplenup{\steplenpot}{v}$. Then, there is a unique $\alpha\in (\steplenlo{\steplenpot}{v},\steplenup{\steplenpot}{v})$ such that $\udim(\levelset^\alpha_{\steplenpot,v})=\udim(\shiftsp)$ \cite{pesinweiss}. For this $\alpha$, by Theorem \ref{thm:main:freq} and Theorem \ref{thm:main:notfreq}, we have that for any finite set $\wordset$ of admissible words,
            \begin{align*}
                \udim(\shiftsp)= \udim(\levelset^\alpha_{\steplenpot,v})=\udim(\uniflevset^\alpha_{\steplenpot,v}\cap\freqappearset_\wordset)&\leq\udim(\freqappearset_\wordset)\leq\udim(\shiftsp);\\
                \udim(\shiftsp)= \udim(\levelset^\alpha_{\steplenpot,v})=\udim(\uniflevset^\alpha_{\steplenpot,v}\cap\notfreq_\wordset)&\leq\udim(\notfreq_\wordset)\leq\udim(\shiftsp).
            \end{align*}
            This completes the proof.
        \end{proof}
    \subsection{Remarks on Boundary Case}
        In this section, we focus on the boundary case, in which $\alpha\in\set{\steplenlo{\steplenpot}{v},\steplenup{\steplenpot}{v}}$. We shall first show that Theorem \ref{thm:main:freq} in general fails for the boundary case. This will be achieved by showing Proposition \ref{prop:counterex} below; see Remark \ref{rmk:counterex}. After that, we shall show that $\udim(\uniflevset^\alpha_{ \steplenpot,v})=\udim(\levelset^\alpha_{\steplenpot,v})$ holds in general, even for the boundary case. This means that we will prove Proposition~\ref{prop:main:rec}.
        \begin{prop} \label{prop:counterex}
            Let $\lshift:\shiftsp\rightarrow\shiftsp$ be a topologically transitive SFT. Let the functions $\steplenpot:\shiftsp\rightarrow\bbR$ and $v:\shiftsp\rightarrow(-\infty,0)$ be \holder continuous functions satisfying that $\steplenlo{\steplenpot}{v}<\steplenup{\steplenpot}{v}$. Then, for any $\alpha\in\set{\steplenlo{\steplenpot}{v},\steplenup{\steplenpot}{v}}$, there is a set $\wordset$ containing only one admissible word satisfying $\uniflevset^\alpha_{\steplenpot,v}\cap\freqappearset_\wordset=\varnothing$.
        \end{prop}
        \begin{rmk} \label{rmk:counterex}
            Pick $\steplenpot$ such that $\steplenlo{\steplenpot}{v}<\steplenup{\steplenpot}{v}$, $\alpha\in\set{\steplenlo{\steplenpot}{v},\steplenup{\steplenpot}{v}}$ and $\udim(\levelset^\alpha_{\steplenpot,v})>0$. Such a choice of $\varphi$ is possible \cite{schmeling}. Then, we immediately get a counterexample for the claim of Theorem \ref{thm:main:freq} for the boundary case.
        \end{rmk}
        \begin{proof}[Proof of Proposition \ref{prop:counterex}]
            We only prove the proposition for the case where $\alpha=\steplenlo{\steplenpot}{v}$. When $\alpha=\steplenup{\steplenpot}{v}$, one can use a similar argument to prove the same claim.
            
            Assume $\alpha=\steplenlo{\steplenpot}{v}$. Since $\steplenup{\steplenpot}{v}>\steplenlo{\steplenpot}{v}=\alpha$, we have $\inf_{n\geq1} S_n\steplenpot_\alpha(\xi^*)=-\infty$ for some $\xi^*\in\shiftsp$. Take a positive integer $n$ such that
            \begin{equation} \label{eqn:prop:counterex:proof1}
                S_n\steplenpot_\alpha (\xi^*)<-\distconst_{\steplenpot_\alpha}-\|\infixset\|\cdot\|\steplenpot_\alpha\|-C_--1,
            \end{equation}
            where $C_-\defeq\sup_{m\geq1}\sup_{\zeta\in\shiftsp}S_m\steplenpot_\alpha(\zeta)$ is finite due to Corollary \ref{cor:bd}. Set
            \begin{equation*}
                \omega\defeq \xi^*_1\cdots\xi^*_n\rho,
            \end{equation*}
            where $\rho$ is an admissible word in $\infixset$ that makes $\omega^2=\omega\omega$ admissible. It then follows from (\ref{eqn:prop:counterex:proof1}) that
            \begin{equation} \label{eqn:prop:counterex:proof2}
                \sup_{\zeta \in[\omega]}S_{|\omega|}\steplenpot_\alpha(\zeta)\leq S_{|\omega|}\steplenpot_\alpha(\xi^*)+\distconst_{\steplenpot_\alpha}<-C_--1.
            \end{equation}
            Given any $\xi\in\freqappearset_\omega$, we shall show that $\xi\notin\uniflevset^\alpha_{\steplenpot,v}$, which will complete the proof.
            
            Since $\xi\in\freqappearset_\omega$, we can find a strictly increasing sequence $(n_k)_{k\geq1}$ of positive integers such that $\xi\in\bigcap_{k=1}^\infty\lshift^{-n_k}[\omega]$ and $n_j-n_{j-1}>|\omega|$ for each positive integer $j$, where we set $n_0\defeq 0$. Then, for any integer $k\geq 2$, using (\ref{eqn:prop:counterex:proof2}) and the definition of $C_-$, we have
            \begin{align*}
                S_{n_k} \steplenpot_\alpha(\xi)&\leq \sum_{j=1}^kS_{n_j-n_{j-1}}\steplenpot_\alpha(\lshift^{n_{j-1}}(\xi))\\
                &\leq C_-+\sum_{j=2}^kS_{|\omega|}\steplenpot_\alpha(\lshift^{n_{j-1}}(\xi))+S_{n_j-n_{j-1}-|\omega|}\steplenpot_\alpha(\lshift^{n_{j-1}+|\omega|}(\xi))\\
                &< C_-+\sum_{j=2}^k \left(-C_--1+C_-\right)=C_--k+1.
            \end{align*}
            We can thus conclude that $\xi\notin\uniflevset^\alpha_{\steplenpot,v}$.
        \end{proof}
        \begin{proof}[Proof of Proposition \ref{prop:main:rec}]
            Without loss of generality, we may assume that $\steplenlo{\steplenpot}{v}\leq\alpha\leq\steplenup{\steplenpot}{v}$, for otherwise $\uniflevset^\alpha_{\steplenpot,v}$ and $\levelset^\alpha_{\steplenpot,v}$ would both be empty. When $\steplenlo{\steplenpot}{v}=\steplenup{\steplenpot}{v}=\alpha$, by Corollary \ref{cor:bd}, we have $\uniflevset^\alpha_{\steplenpot,v}=\levelset^\alpha_{\steplenpot,v}=\shiftsp$, so the claim holds trivially. Hence, we shall henceforth assume that  $\steplenlo{\steplenpot}{v}<\steplenup{\steplenpot}{v}$. If $\steplenlo{\steplenpot}{v}<\alpha <\steplenup{\steplenpot}{v}$, note that $\uniflevset^\alpha_{\steplenpot,v}\cap\freqappearset_\wordset\subseteq\uniflevset^\alpha_{\steplenpot,v}\subseteq\levelset^\alpha_{\steplenpot,v}$, where $\wordset$ is an arbitrary finite set of admissible words, so our claim is an immediate consequence of Theorem \ref{thm:main:freq}. Hence, it remains to handle the case where $\steplenlo{\steplenpot}{v}<\steplenup{\steplenpot}{v}$ and $\alpha$ is equal to either $\steplenlo{\steplenpot}{v}$ or $\steplenup{\steplenpot}{v}$.
            
            Suppose that $\alpha=\steplenlo{\steplenpot}{v}$. Take $\gibbsmeas_\alpha$ as in the case (\ref{item:pw:bdy}) of Theorem~\ref{thm:pw}. We claim that $\gibbsmeas_\alpha(\uniflevset^\alpha_{\steplenpot,v})=1$. To see this, take a sub-action $g_-:\shiftsp\rightarrow\bbR$ satisfying that $\steplenpot_\alpha+g_-\circ\lshift-g_-\leq0$; Theorem \ref{thm:subact} guarantees the existence of $g_-$. Note that by Theorem \ref{thm:pw}, we have $\gibbsmeas_\alpha(\levelset^\alpha_{\steplenpot,v})=1$, implying that
            \begin{equation*}
                \lim_{ n\rightarrow+\infty}\frac{1}{n}S_n\steplenpot_\alpha=0,\,\gibbsmeas_\alpha\text{-a.e.}
            \end{equation*}
            Hence, by Birkhoff's ergodic theorem, we have $\int_{\shiftsp} \steplenpot_\alpha\,\rmd\gibbsmeas_\alpha=0$. The $\lshift$-invariance of $\gibbsmeas_\alpha$ then implies that
            \begin{equation*}
                \int_{\shiftsp} S_n(\steplenpot_\alpha+g_-\circ\lshift-g_-)\,\rmd\gibbsmeas_\alpha=n\int_{\shiftsp}(\steplenpot_\alpha+g_-\circ\lshift-g_-)\,\rmd\gibbsmeas_\alpha=n \int_{\shiftsp} \steplenpot_\alpha\,\rmd\gibbsmeas_\alpha=0,
            \end{equation*}
            for every positive integer $n$. For any $n$, since $S_n(\steplenpot_\alpha+g_-\circ\lshift-g_-)\leq 0$, we have
            \begin{equation*}
                S_n\steplenpot_\alpha +g_-\circ\lshift^n-g_-=S_n(\steplenpot_\alpha+g_-\circ\lshift-g_-)=0,\,\gibbsmeas_\alpha\text{-a.e.}
            \end{equation*}
            It follows that $|S_n\steplenpot_\alpha|\leq 2\|g_-\|$, $\gibbsmeas_\alpha$-a.e. Therefore, we have $\gibbsmeas_\alpha(\uniflevset^\alpha_{\steplenpot,v})=1$. From $\udim(\gibbsmeas_\alpha)=\udim(\levelset^\alpha_{\steplenpot,v})$ in Theorem \ref{thm:pw}, we deduce that $\udim(\uniflevset^\alpha_{\steplenpot,v})=\udim(\levelset^\alpha_{\steplenpot,v})$ when $\alpha=\steplenlo{\steplenpot}{v}$. By symmetry, we also have the same result when $\alpha=\steplenup{\steplenpot}{v}$. The proof is thus complete.
        \end{proof}

\section{Application to Gibbs Measure on One-Dimensional Manifold}\label{sec:app}
    \subsection{Setup} \label{subsec:app:setup}
        Recall that $\onedimmani$ denotes either the circle $\bbS^1$ or the real line $\bbR$. We begin with the definition of a conformal graph directed Markov system (CGDMS) on $\onedimmani$. The definition of CGDMS on a more general space is given in \cite{gdms}.
        
        Let $(\vertexset,\edgeset)$ be a directed graph, where $\vertexset$ is the set of vertices and $\edgeset$ is the set of directed edges. It is assumed in \cite{gdms} that $\vertexset$ is finite, and $\edgeset$ is countable (possibly finite). In this paper, we shall always assume that $\edgeset$ is finite and contains at least two elements. Define $\initialv,\terminalv:\edgeset\rightarrow\vertexset$ as the mappings sending every directed edge $e\in\edgeset$ to its initial vertex and terminal vertex respectively. There is a natural incidence matrix $\incmat:\edgeset\times\edgeset\rightarrow\set{0,1}$ arising from the graph; let $\incmat(e,e')=1$ if and only if $\terminalv(e)=\initialv(e')$. Let $\lshift:\shiftsp\rightarrow\shiftsp$ denote the SFT for which the set of symbols is $\edgeset$ and the incidence matrix is $\incmat$. We assume that the directed graph $(\vertexset,\edgeset)$ is strongly connected, meaning that for any ordered pair $(p,p')$ of vertices, there exists a path starting from $p$ and ending at $p'$. It follows from this assumption that the SFT $\lshift:\shiftsp\rightarrow\shiftsp$ is topologically transitive.

        Let $\intervalfamily=\set{\interval_p|p\in\vertexset}$ be a family of compact intervals of positive lengths in $\onedimmani$, whose interiors are pairwise disjoint. Let $\GDMS\defeq\set{\contraction_e:\interval_{\terminalv(e)}\rightarrow\interval_{\initialv(e)}|e\in\edgeset}$ be a family of contractions. Following \cite{gdms}, we call $\GDMS$ a graph directed Markov system (GDMS). Since $\edgeset$ is finite, this means that there is some positive constant $c<1$ such that $d_{\onedimmani}(\contraction_e(x),\contraction_e(y))\leq c\cdot d_{\onedimmani}(x,y)$ for any $e\in\edgeset$ and any $x,y\in\interval_{\terminalv(e)}$.
        
        For each $e\in\edgeset$, we assume that $\contraction_e$ can be extended to a $C^1$ diffeomorphism in some open neighbourhood of $\interval_{\terminalv(e)}$, say $U_e$. The derivatives $\contraction'_e$ on $U_e$ is assumed to be \holder continuous. Finally, we assume the open set condition. For any $E\subseteq\onedimmani$, let $\interior(E)$ denote the interior of $E$. Then the open set condition means that
        \begin{equation*}
            \interior(\contraction_e (\interval_{\terminalv(e)}))\cap\interior(\contraction_{e'} (\interval_{\terminalv(e')}))=\varnothing,
        \end{equation*}
        for any two distinct $e,e'\in\edgeset$. Under all these assumptions, the GDMS $\GDMS$ satisfies all the axioms of a conformal graph directed Markov system (CGDMS); see \cite{gdms} for details.
        
        For any admissible word $\omega$ over $\edgeset$, define
        \begin{equation*}
            \contraction_\omega \defeq\contraction_{\omega_1}\circ\cdots\circ\contraction_{\omega_{|\omega|}}.
        \end{equation*}
        Denote $\omega_1$ and $\omega_{|\omega|}$ by $a$ and $b$. Then, clearly, $\contraction_\omega$ is a map defined on $\interval_{\terminalv(b)}$. Define $\interval_\omega\defeq\contraction_\omega(\interval_{\terminalv(b)})$. The coding map $\coding:\shiftsp\rightarrow\onedimmani$ is defined by setting $\coding(\xi)$ as the unique point in the singleton $\bigcap_{k=1}^\infty\interval_{\xi_1\cdots\xi_k}$. The limit set is then the set $\attractor\defeq\coding(\shiftsp)$.
        
        We define $\geopot:\shiftsp\rightarrow\bbR$ by $\geopot(\xi)\defeq -\log|\contraction'_{\xi_1}(\coding(\xi))|$ for every $\xi\in\shiftsp$. As the contractions $\contraction_e$ are extended to $C^1$ diffeomorphisms, $\contraction'_{\xi_1}(\coding(\xi))$ cannot be zero, so $\geopot$ takes finite values. By our assumptions above, we can readily see that $\geopot$ is \holder continuous and strictly positive. The \holder continuity of the derivatives of the contractions implies that
        \begin{equation*}
            \sup_{\omega\in \symbset^*_\incmat}\sup_{x,y\in\interval}\left|\frac{\contraction'_\omega(x)}{\contraction'_\omega(y)}\right|<+\infty.
        \end{equation*}
        We refer to \cite{gdms} for the proof of this well-known fact. As a consequence, from the observation that $S_n\geopot(\xi)=-\log|\contraction'_{\xi_1\cdots\xi_n}(\coding(\xi))|$ for each positive integer $n$, we have
        \begin{equation} \label{eqn:contr'dist}
            \sup_{n\in\bbN} \sup_{\xi\in\shiftsp}\sup_{x\in\interval_{\terminalv(\xi_n)}}\Big|-S_n\geopot(\xi)-\log\left|\contraction'_{\xi_1\cdots\xi_n}(x)\right|\Big|<+\infty.
        \end{equation}
        
        Suppose that we are given a Gibbs measure $\gibbsmeas$ on $\shiftsp$ of some \holder continuous function $\steplenpot:\shiftsp\rightarrow\bbR$, whose topological pressure $P(\steplenpot)$ equals zero. Let $C_\gibbsmeas\geq 1$ be a constant such that for any admissible word $\omega$ over $\symbset$,
        \begin{equation*}
            C_\gibbsmeas^{-1} \exp(S_\omega\steplenpot)\leq\gibbsmeas([\omega])\leq C_\gibbsmeas\exp(S_\omega\steplenpot).
        \end{equation*}
        The Gibbs measure $\gibbsmeas$ can be transferred to a measure supported on the limit set $\attractor$, namely the pushforward measure $\coding_*\gibbsmeas$ through the coding map $\coding$.
    \subsection{Multifractal Analysis of H\"older Regularity}
        We begin with the construction of a set $\wordset$ of admissible words for which (\ref{eqn:prop:der}) is satisfied.
        
        For every edge $e\in\edgeset$, we have seen that $\interval_e\subsetneq\onedimmani$, so we can take a total ordering $\leq_e$ for each $\interval_e$. For each $e\in\edgeset$, define $P_e\defeq\set{\min(\interval_e\cap\attractor),\max(\interval_e\cap\attractor)}$, where the minimum and maximum are taken with respect to $\leq_e$. Define $P\defeq\bigcup_{e\in\edgeset}P_e$. Also recall that for any $e\in\edgeset$ and any $x,y\in\interval_e$, $[x,y]_e$ denotes the unique subinterval of $\interval_e$ of which the two endpoints are $x$ and $y$.

        \begin{lma} \label{lma:app:wordset}
            The set $\coding^{-1}(P)$ is finite. Moreover, there exists a non-negative integer $N$ such that for every $\xi\in\coding^{-1}(P)$, $\lshift^N(\xi)$ is a periodic sequence.
        \end{lma}
        \begin{proof}
            Note that for any two distinct admissible words $\omega,\tau$ of the same length, the open set condition implies that $\interior(\interval_\omega)\cap\interior(\interval_\tau)=\varnothing$. It follows that, for any $x\in\attractor$, $\coding^{-1}\set{x}$ has no more than two elements. Therefore, we have
            \begin{equation*}
                \#\coding^{-1}(P)\leq 2\cdot\# P\leq 4\cdot\#\vertexset<+\infty.
            \end{equation*}
            
            To show the second claim, we need to prove the following claim:
            \begin{equation} \label{eqn:lma:app:wordset}
                \lshift( \coding^{-1}P)\subseteq \coding^{-1}(P).
            \end{equation}
            Take $\xi\in\coding^{-1}(P)$ arbitrarily. As $\contraction_{\xi_1}:\interval_{\xi_2}\rightarrow\interval_{\xi_1\xi_2}$ is a $C^1$ diffeomorphism, it either preserves or reverses the ordering. Note that $\contraction_{\xi_1}(\coding(\lshift(\xi)))= \coding(\xi)$ and $\coding(\xi)\in P\cap \interval_{\xi_1}=P_{\xi_1}$, so $\contraction_{\xi_1}(\coding(\lshift(\xi))$ is either the minimum or the maximum of $\interval_{\xi_1}\cap\attractor$ under the total ordering $\leq_{\xi_1}$. Therefore, by the monotonicity of $\contraction_{\xi_1}$ we have $\coding(\lshift(\xi))\in P_{\xi_2}\subseteq P$. Since $\xi$ is arbitrarily taken from $\coding^{-1}(P)$, we conclude that $\lshift(\coding^{-1}(P))\subseteq\coding^{-1}(P)$.

            Combining (\ref{eqn:lma:app:wordset}) with the fact that $\coding^{-1}(P)$ is a finite set, we see that for each $\xi\in\coding^{-1}(P)$, we can take a non-negative integer $N_\xi$ such that $\lshift^{N_\xi}(\xi)$ is periodic. Take an arbitrary $N\geq\max_{\xi\in\coding^{-1}(P)}$. This $N$ then satisfies the requirement in the second claim.
        \end{proof}

        Take a non-negative integer $N$ satisfying the condition in Lemma \ref{lma:app:wordset}. Then, for each $\xi\in\coding^{-1}(P)$, we are able to take one word $\omega(\xi)$ such that $\lshift^N(\xi)=\omega(\xi)\omega(\xi)\cdots$. Henceforth, we always set
        \begin{equation*}
            \wordset\defeq \set{\omega(\xi)|\xi\in\coding^{-1}(P)}.
        \end{equation*}
        Clealy, $\wordset$ is a finite set of admissible words.

        For this $\wordset$ and $v=-\geopot$, we can now prove the inclusion relation (\ref{eqn:prop:der}).
        \begin{prop} \label{prop:der}
           For any $\alpha\geq 0$, we have  $\uniflevset^\alpha_{ \steplenpot,v}\cap\notfreq_\wordset\subseteq\coding^{-1}\left(\moderate^\alpha_\gibbsmeas\right)\subseteq\uniflevset^\alpha_{\steplenpot,v}$.
        \end{prop}
        A similar assertion for a more restricted situation appears in \cite[Proposition 2.3]{kessestratholder}.
        \begin{proof}
            Fix $\alpha\geq0$. As in the previous section, define $\steplenpot_\alpha\defeq \steplenpot-\alpha v=\steplenpot+\alpha\geopot$.
            
            We begin with the first inclusion in the proposition. Without loss of generality, assume that $\uniflevset^\alpha_{\steplenpot,v}\cap\notfreq_\wordset$ is non-empty and take an arbitrary $\xi\in\uniflevset^\alpha_{\steplenpot,v}\cap\notfreq_\wordset$. Now that $\xi\in\notfreq_\wordset$, we have $\lshift^n\xi \notin \coding^{-1}(P)$ for any non-negative integer $n$. In particular, $\coding(\xi)$ must lie in the interior of every $\interval_{\xi_1\cdots\xi_n}$.
            
            Now fix an arbitrary $y\neq \coding(\xi)$ in $\interval_{\xi_1}$. Since the set sequence $(\interval_{\xi_1\cdots\xi_n})_{n\geq 1}$ is descending and $\bigcap_{n=1}^\infty\interval_{\xi_1\cdots\xi_n}=\set{\coding(\xi)}$, there exists a unique positive integer $m$ such that $y\in\interval_{\xi_1\cdots\xi_m}\setminus\interval_{\xi_1\cdots\xi_{m+1}}$.
            
            Note that $\interval_{\xi_1\cdots\xi_m}$ is connected, so on the one hand, we can readily see that
            \begin{equation} \label{eqn:prop:der:proof1}
                [\coding(\xi), y]_{\xi_1}\subseteq\interval_{\xi_1\cdots\xi_m}.
            \end{equation}
            On the other hand, take $l$ to be a positive integer such that $\xi\in\notfreq_{\wordset,l}$. Let $N$ be a non-negative integer such that for any $\xi'\in\coding^{-1}(P)$, $\lshift^N(\xi')$ is periodic. The existence of this $N$ is guaranteed by Lemma \ref{lma:app:wordset}. Define
            \begin{equation*}
                L\defeq (l+1)\|\wordset\|+N.
            \end{equation*}
            We claim that there is an admissible word $\tau$ over $\edgeset$ of length $L$ such that $\xi_{m+1}\tau$ is admissible and
            \begin{equation} \label{eqn:prop:der:proof2}
                \interval_{\xi_1 \cdots\xi_{m+1}\tau}\subseteq[\coding(\xi),y]_{\xi_1}.
            \end{equation}
            
            To prove this claim, first note that $\contraction_{\xi_1\cdots\xi_{m-1}}:\interval_{\xi_m}\rightarrow\interval_{\xi_1\cdots\xi_m}$ is invertible. Set $y'\defeq\contraction^{-1}_{\xi_1\cdots\xi_{m-1}}(y)$ and $\xi'\defeq\lshift^{m-1}(\xi)$. Define
            \begin{equation*}
                z\defeq \max(\attractor\cap\interval_{\xi_m\xi_{m+1}})
            \end{equation*}
            if $\coding(\xi')\leq_{\xi_m}y'$. Otherwise, define
            \begin{equation*}
                z\defeq \min(\attractor\cap\interval_{\xi_m\xi_{m+1}}).
            \end{equation*}
            Here, the maximum and minimum are both taken with respect to $\leq_{\xi_m}$. Take $\zeta\in[\xi_m\xi_{m+1}]$ satisfying $\coding(\zeta)=z$. By the monotonicity of $\contraction_{\xi_m}$, we see that $\coding(\lshift(\zeta))$ is in $P_{\xi_{m+1}}$. It follows that $\lshift^{N+1}(\zeta)$ is periodic, and furthermore, there exists an integer $k\in\set{1,\cdots,\|\wordset\|}$ and a word $\omega\in\wordset$ such that $\lshift^{N+1+k}(\zeta)=\omega\omega\cdots$. Take $\tau$ as the unique admissible word of length $L$ such that $\zeta\in[\xi_m\xi_{m+1}\tau]$. Then, on the one hand, we have that $\tau$ contains $\omega^l$ as a subword. On the other hand, since $\xi\in\notfreq_{\wordset,l}$, the word $\xi_{m+2}\cdots\xi_{m+L+1}$ does not contain $\omega^l$ as a subword. Therefore, $\xi_m\xi_{m+1}\tau$ and $\xi_m\cdots\xi_{m+L+1}$ are distinct words of length $L+2$, so their cylinder sets are disjoint. As a result, we have
            \begin{equation*}
                \coding(\xi')\notin \interval_{\xi_m\xi_{m+1}\tau}.
            \end{equation*}
            As $\interval_{\xi_m\xi_{m+1}\tau}$ is connected, $\interval_{\xi_m\xi_{m+1}\tau}$ lies on one side of $\coding(\xi')$. By the same reason, because $y'\notin\interval_{\xi_m\xi_{m+1}}$, the interval $\interval_{\xi_m\xi_{m+1}\tau}$ also lies on only one side of $y'$. By our definition of $z$, we have $z\in[y',\coding(\xi')]_{\xi_m}$, so $[y',\coding(\xi')]_{\xi_m}\cap\interval_{\xi_m\xi_{m+1}\tau}\neq\varnothing$. Combining this with the fact that $\interval_{\xi_m\xi_{m+1}\tau}$ is on one side of $\coding(\xi')$ and also on one side of $y'$, we deduce that
            \begin{equation*}
                \interval_{\xi_m \xi_{m+1}\tau}\subseteq[y',\coding(\xi')]_{\xi_m}.
            \end{equation*}
            Applying $\contraction_{\xi_1\cdots\xi_{m-1}}$ on both sides, we have
            \begin{align*}
                \interval_{\xi_1 \cdots\xi_{m+1}\tau}&=\contraction_{\xi_1\cdots\xi_{m-1}}(\interval_{\xi_m\xi_{m+1}\tau})\\
                &\subseteq [\contraction_{\xi_1\cdots\xi_{m-1}}(y'),\contraction_{\xi_1\cdots\xi_{m-1}}(\coding(\xi'))]_{\xi_1}=[y,\coding(\xi)]_{\xi_1}.
            \end{align*}
            Note that the monotonicity of $\contraction_{\xi_1\cdots\xi_{m-1}}$ is used here. Therefore, $\tau$ satisfies (\ref{eqn:prop:der:proof2}).
            
            By (\ref{eqn:prop:der:proof1}) and (\ref{eqn:prop:der:proof2}), we have
            \begin{equation*}
                \coding_*\gibbsmeas ([y,\coding(\xi)]_{\xi_1})\leq\coding_*\gibbsmeas(\interval_{\xi_1\cdots\xi_m})\leq C_\gibbsmeas\exp(S_{\xi_1\cdots\xi_m}\steplenpot)\leq C_\gibbsmeas\exp(\distconst_\steplenpot)\exp(S_m\steplenpot(\xi)),
            \end{equation*}
            and
            \begin{align*}
                \coding_*\gibbsmeas ([y,\coding(\xi)]_{\xi_1})\geq\coding_*\gibbsmeas(\interval_{\xi_1\cdots\xi_m\tau})&\geq C_\gibbsmeas^{-1}\exp(S_{\xi_1\cdots\xi_m\tau}\steplenpot)\\
                &\geq C_\gibbsmeas^{-1}\exp(-L\|\steplenpot\|-\distconst_\steplenpot)\exp(S_m\steplenpot(\xi)).
            \end{align*}
            In what follows, let $y\neq\coding(\xi)$ be sufficiently close to $\coding(\xi)$ such that the shortest path connecting $\coding(\xi)$ and $y$ is exactly $[y,\coding(\xi)]_{\xi_1}$. Then, by the mean value theorem and (\ref{eqn:contr'dist}), there exists a constant $C_1\geq 1$ such that
            \begin{align*}
                \diam( \interval_{\xi_1 \cdots\xi_m})&\leq C_1\exp(-S_m\geopot(\xi));\\
                \diam( \interval_{\xi_1 \cdots\xi_{m}\tau})&\geq C_1^{-1}\exp(-S_m\geopot(\xi)-L\|\geopot\|).
            \end{align*}
            Combining these with (\ref{eqn:prop:der:proof1}) and (\ref{eqn:prop:der:proof2}), we have
            \begin{equation*}
                C_1^{-1} \exp(S_nv(\xi)-L\|v\|)\leq d_{\onedimmani}(\coding(\xi),y) \leq C_1\exp(S_nv(\xi)).
            \end{equation*}
            Consequently, there exists a constant $C\geq 1$ such that for $y\neq\coding(\xi)$ sufficiently close to $\coding(\xi)$,
            \begin{align}
                C^{-1} \exp(S_n\steplenpot(\xi))&\leq\coding_*\gibbsmeas([\coding(\xi),y]_{\xi_1})\leq C\exp(S_n\steplenpot(\xi));\\
                C^{-1}\exp(-S_n\geopot(\xi))&\leq d_{\onedimmani}(\coding(\xi),y)\leq C\exp(-S_n\geopot(\xi)).\label{eqn:prop:der:proof:bilip}
            \end{align}
            From this we see that
            \begin{equation*}
                C^{-2} \exp(S_n\steplenpot_\alpha(\xi))\leq\lodert^\alpha\cdf(\coding(\xi))\leq\updert^\alpha\cdf(\coding(\xi)) \leq C^2\exp(S_n\steplenpot_\alpha(\xi)).
            \end{equation*}
            As $\xi$ is taken from $\uniflevset^\alpha_{\steplenpot,v}$, we have $\coding(\xi)\in\moderate^\alpha_\gibbsmeas$. This shows the first half of (\ref{eqn:prop:der}).
            
            Now it only remains to show the second half of (\ref{eqn:prop:der}). If $\coding^{-1}(\moderate^\alpha_\gibbsmeas)$ is empty, then the claim is trivial. Otherwise, take $\zeta\in\coding^{-1}(\moderate^\alpha_\gibbsmeas)$ arbitrarily.
            
            We define a sequence $(y_n)_{n\geq1}$ of points in $\interval$ in the following manner. For any $n\geq1$, let $\omega^{(n)}$ be an admissible word of length $m(n)\leq\#\symbset$ satisfying that $\omega^{(n)}_1=\zeta_{n+1}$ and $\omega^{(n)}_{m(n)}\neq\zeta_{n+m(n)}$. By construction, we have $\coding(\zeta)\notin\interval_{\zeta_1\cdots\zeta_n\omega^{(n)}}$, so there is a unique endpoint $y_n$ of $\interval_{\zeta_1\cdots\zeta_n\omega^{(n)}}$ satisfying that $\interval_{ \zeta_1\cdots \zeta_n\omega^{(n)}}\subseteq\left[\coding(\zeta),y_n\right]_{\zeta_1}$. Since $\coding(\zeta)$ and $y_n$ are both in $\interval_{\zeta_1\cdots\zeta_n}$, we also have $\left[\coding(\zeta),y_n\right]_{\zeta_1}\subseteq\interval_{\zeta_1\cdots\zeta_n}$. Therefore, there is some constant $C'\geq 1$ such that for all $n\in\bbN$,
            \begin{equation*}
                \frac{1}{C'}\exp(S_n \steplenpot_\alpha(\zeta))\leq \frac{\coding_*\gibbsmeas([\coding(\zeta),y_n]_{\zeta_1})}{|\coding(\zeta)-y_n|^\alpha}\leq C'\exp(S_n\steplenpot_\alpha(\zeta)).
            \end{equation*}
            Clearly $\lim_{n\rightarrow\infty}y_n=\coding(\zeta)$, so the inequality above implies that $\zeta\in\uniflevset^\alpha_{\steplenpot,v}$. Hence, we may conclude that the second inclusion in (\ref{eqn:prop:der}) holds as well.
        \end{proof}
        Now we are ready to give the proof of Theorem \ref{thm:app}, mainly using Theorem \ref{thm:main:notfreq} and Proposition \ref{prop:der}.
        \begin{proof}[Proof of Theorem \ref{thm:app}]
            Endow $\shiftsp$ with the metric $\symbdist_\geopot$ given by $\geopot$. Then, the coding map $\coding:\shiftsp\rightarrow\attractor$ is Lipschitz. Hence, on the one hand, by Proposition \ref{prop:main:rec} and Proposition \ref{prop:der}, we have
            \begin{equation*}
                \udim( \levelset^\alpha_{\steplenpot,v})=\udim(\uniflevset^\alpha_{\steplenpot,v})\geq\udim(\coding^{-1}(\moderate^\alpha_\gibbsmeas))\geq\hausdim(\moderate^\alpha_\gibbsmeas).
            \end{equation*}
            On the other hand, for any integer $l\geq 1$, if we restrict $\coding$ to $\notfreq_{\wordset,l}$, it becomes injective and further bi-Lipschitz onto its image, which can be seen from (\ref{eqn:prop:der:proof:bilip}). Therefore, by the countable stability of Hausdorff dimension,
            \begin{align*}
                \udim \left(\uniflevset^\alpha_{\steplenpot,v}\cap\notfreq_\wordset\right)&=\sup_{l\geq1}\udim\left(\uniflevset^\alpha_{\steplenpot,v}\cap\notfreq_{\wordset,l}\right)\\
                &= \sup_{l\geq1}\hausdim\left(\coding\left(\uniflevset^\alpha_{\steplenpot,v}\cap\notfreq_{\wordset,l}\right)\right)\leq\hausdim\left(\moderate^\alpha_\gibbsmeas\right).
            \end{align*}
            By Theorem \ref{thm:main:notfreq}, when $\alpha\notin\set{\steplenlo{\steplenpot}{v},\steplenup{\steplenpot}{v}}$, all the inequalities above are equalities, so $\hausdim(\moderate^\alpha_\gibbsmeas)=\udim(\levelset^\alpha_{\steplenpot,v})$.
        \end{proof}
        
        Finally, we are now able to give the proof of Corollary \ref{cor:maxdim} as follows.
        \begin{proof}[Proof of Corollary \ref{cor:maxdim}]
            First consider the case where $\steplenlo{\steplenpot}{v}<\steplenup{\steplenpot}{v}$. Then there is a unique $\alpha_0\in(\steplenlo{\steplenpot}{v},\steplenup{\steplenpot}{v})$ such that $\udim(\levelset^{\alpha_0}_{\steplenpot,v})=\udim(\shiftsp)$ \cite{pesinweiss}. By Theorem \ref{thm:app}, we thus have
            \begin{equation*}
                \hausdim( \moderate^{\alpha_0}_\gibbsmeas)=\udim(\levelset^{\alpha_0}_{\steplenpot,v})=\udim(\shiftsp)=\hausdim(\attractor),
            \end{equation*}
            and for any $\alpha\neq\alpha_0$,
            \begin{equation*}
                \hausdim( \moderate^\alpha_\gibbsmeas)\leq\udim(\levelset^\alpha_{\steplenpot,v})<\udim(\shiftsp)=\hausdim(\attractor).
            \end{equation*}
            This completes the proof for the case where $\steplenlo{\steplenpot}{v}<\steplenup{\steplenpot}{v}$.

            When $\steplenlo{\steplenpot}{v}=\steplenup{\steplenpot}{v}$, take $\alpha_0$ to be $\steplenlo{\steplenpot}{v}=\steplenup{\steplenpot}{v}$. Then we only need to show that
            \begin{equation} \label{eqn:cor:maxdim}
                \hausdim( \moderate^{\alpha_0}_\gibbsmeas)=\hausdim(\attractor),
            \end{equation}
            because for any $\alpha\neq\alpha_0$, by Proposition \ref{prop:der}, we have $\moderate^\alpha_\gibbsmeas\subseteq\coding(\levelset^\alpha_{\steplenpot,v})=\varnothing$. To show (\ref{eqn:cor:maxdim}), first note that by Proposition \ref{prop:der} and Corollary \ref{cor:bd}, we have
            \begin{equation*}
                \notfreq_\wordset= \shiftsp\cap\notfreq_\wordset=\uniflevset^{\alpha_0}_{ \steplenpot,v}\cap\notfreq_\wordset\subseteq\coding^{-1}(\moderate^{\alpha_0}_\gibbsmeas).
            \end{equation*}
            Also recall that we have seen in Corollary \ref{cor:freqnotfreqmaxdim} that $\udim(\notfreq_\wordset)=\udim(\shiftsp)$, so
            \begin{equation*}
                \hausdim( \moderate^{\alpha_0}_\gibbsmeas)=\udim(\coding^{-1}(\moderate^{\alpha_0}_\gibbsmeas))\geq\udim(\shiftsp).
            \end{equation*}
            Our proof is thus complete.
        \end{proof}
\appendix
\section{Proofs of Two Lemmas}\label{app:key}
    We shall prove Lemma \ref{lma:gjk} and Lemma \ref{lma:infixset} in this section. The notations we used in Section \ref{sec:dimspecunif} will also be used in the proofs below.

    \begin{proof}[Proof of Lemma \ref{lma:infixset}]
        In this proof, $N$ will denote $\#\symbset$, the cardinality of $\symbset$. For positive integers $n,k$ and symbols $a,b\in\symbset$, we define
        \begin{equation*}
            \mathcal{J}_{n,k}(a,b)\defeq\Set{\rho\in\symbset^*_\incmat|n\leq|\rho|\leq n+k\text{ and the word }a\rho b\text{ is admissible}}.
        \end{equation*}
        Fix an arbitrary integer $n\geq 1$ and $a,b\in\symbset$. We shall construct a map $\eta^n_{a,b}:\symbset^n_\incmat\rightarrow\mathcal{J}_{n,2N}(a,b)$ as follows.
        
        For every $\omega\in\symbset^n_\incmat$, pick $\tau$ and $\theta$ to be two shortest words, possibly empty, for which $a\tau\omega\theta b$ is admissible, and define $\eta^n_{a,b}(\omega)\defeq \tau\omega\theta$. Note that for any two symbols $a,b\in\symbset$, the length of the shortest word $\rho$ for which $a\rho b$ is admissible cannot exceed $N$. Therefore, the lengths of $\tau$ and $\theta$ are no greater than $N$. Thus, we have $n\leq |\eta^n_{a,b}(\omega)|\leq n+2N$, justifying that $\eta^n_{a,b}$ maps $\symbset^n_\incmat$ to $\mathcal{J}_{n,2N}(a,b)$.
        
        Given any $\bar{\omega}\in\eta^n_{a,b}(\symbset^n_\incmat)$, note that any $\omega\in(\eta^n_{a,b})^{-1}\set{\bar{\omega}}$, as a subword of $\bar{\omega}$, must start at one of the first $N+1$ entries in $\bar{\omega}$, so $\# (\eta^n_{a,b})^{-1}\set{\bar{\omega}}\leq N+1$. As a consequence, we have
        \begin{equation*}
            \#\mathcal{J}_{n,2N}(a,b) \geq \frac{\#\symbset^n_\incmat}{N+1}.
        \end{equation*}
        Since we assumed that $\udim(\shiftsp)>0$, we have
        \begin{equation*}
            \lim_{n \rightarrow+\infty} \frac{1}{n}\log\#\symbset^n_\incmat=h_{top}>0,
        \end{equation*}
        where $h_{top}$ denotes the topological entropy of $\lshift:\shiftsp\rightarrow\shiftsp$. Thus, we can take a sufficiently large $n$ such that $\#\mathcal{J}_{n,2N}(a,b)>C_N$ for any $a,b\in\symbset$, where
        \begin{equation*}
            C_N\defeq(N^2-1)\left(2N+1+\frac{N^{2N+1}-1}{N-1}\right).
        \end{equation*}

        Label the elements of $\symbset$ so that $\symbset\defeq\set{a_1,\cdots,a_N}$. This gives a lexicographic ordering of $\symbset^2\defeq\set{(a,b)|a,b\in\symbset}$. Following this ordering, at some step, we move to $(a_i,a_j)\in\symbset^2$, and pick one word $\rho_{i,j}$ from $\mathcal{J}_{n,2N}(a_i,a_j)$ in the following manner. At the first step, we arbitrarily pick one word $\rho_{1,1}$ from $\mathcal{J}_{n,2N}(a_1,a_1)$. At the $k$-th step, we have chosen $(k-1)$ words, and move on to some $(a_i,a_j)\in\symbset^2$. When we choose $\rho_{i,j}$ from $\mathcal{J}_{n,2N}(a_i,a_j)$, we avoid all the initial blocks and continuations of all the chosen words. As the lengths of $\rho_{i,j}$ and all the chosen words are between $n$ and $n+2N$, we can estimate the numbers of initial blocks and continuations of the chosen words as follows. Fix one arbitrary chosen word and name it $\rho$. Then, we have
        \begin{equation*}
            \#\set{\text{initial block }\rho'\text{ of }\rho| |\rho'|\geq n}\leq 2N+1
        \end{equation*}
        and
        \begin{align*}
            &\quad\# \set{\text{continuation }\rho''\text{ of }\rho| |\rho''|\leq n+2N}\\
            &=\sum_{l=n}^{n+2N} \#\set{\text{continuation }\rho''\text{ of }\rho| |\rho''|=l}\\
            &\leq \sum_{l=n}^{n+2N}N^{l-n}=\frac{N^{2N+1}-1}{N-1}.
        \end{align*}
        Since the number of chosen words is $k-1\leq N^2-1$, the number of words we shall avoid is at most $C_N$. Now that we have taken $n$ such that $\#\mathcal{J}_{n,2N}(a_i,a_j)>C_N$, we can take one word $\rho_{i,j}$ from $\mathcal{J}_{n,2N}(a_i,a_j)$ avoiding all the initial blocks and continuations of the chosen words.

        Set $\infixset\defeq\set{\rho_{i,j}|(i,j)\in\set{1,\cdots,N}^2}$. This $\infixset$ clearly satisfies the first condition in Lemma \ref{lma:infixset}. In addition, by our construction, for any two distinct words $\rho,\rho'\in\infixset$, $\rho$ is not an initial block of $\rho'$ and vice versa, implying that $[\rho]\cap[\rho']=\varnothing$. Therefore, $\infixset$ also satisfies the second condition in Lemma \ref{lma:infixset}.
    \end{proof}

    \begin{proof}[Proof of Lemma \ref{lma:gjk}]
        In this proof, we use $\phi$ to denote $\steplenpot_\alpha=\steplenpot-\alpha v$. Since we assumed that $\steplenlo{\steplenpot}{v}<\alpha<\steplenup{\steplenpot}{v}$, there exist $\xi^-,\xi^+\in\shiftsp$ and $n^-,n^+\in\bbN$ such that
        \begin{align*}
            S_{n^-} \phi(\xi^-)&<-K'-2\distconst_{\phi}-\|\infixset\|\cdot\|\phi\|\\
            &<K'+2\distconst_{\phi}+\|\infixset\|\cdot\|\phi\|<S_{n^+}\phi(\xi^+).
        \end{align*}
        Define
        \begin{equation*}
            \postfixset_0\defeq \set{\xi^-_1\cdots\xi^-_k|1\leq k\leq n^-}\cup\set{\xi^+_1\cdots\xi^+_k| 1\leq k\leq n^+}\cup\set{\text{empty word}},
        \end{equation*}
        and $\postfixset\defeq\set{\rho\tau\in\symbset^*_\incmat|\rho\in\infixset,\,\tau\in\postfixset_0}$. We claim that $\postfixset$ is the family of words we want. Clearly, $\postfixset$ is a finite set of admissible words, so we only need to show that $\postfixset$ meets the last requirement.
        
        Take an arbitrary $\omega\in\windowword_{\alpha,K'}$. Let $\xi\in[\omega]$. Then, $\left|S_{|\omega|}\phi(\xi)\right|\leq K'$. Consider the case where $S_{|\omega|} \phi(\xi)\geq 0$. Take one $\rho^-\in\infixset$ such that $\omega\rho^-\xi^-$ is admissible. Note that
        \begin{equation*}
            S_{|\omega|+|\rho^-|}\phi(\omega\rho^-\xi^-)\geq S_{|\omega|}\phi(\xi)-\distconst_{\phi}-\|\infixset\|\cdot\|\phi\|\geq-\distconst_{\phi}-\|\infixset\|\cdot\|\phi\|.
        \end{equation*}
        If $S_{|\omega|+|\rho^-|}\phi(\omega\rho^-\xi^-)\leq0$, then the inequality above indicates that $\omega\rho^-\in\windowword_{\alpha,K}$. Otherwise, note that
        \begin{equation*}
            S_{|\omega|+|\rho^-|+n^-}\phi(\omega\rho^-\xi^-)\leq S_{|\omega|}\phi(\xi)+\distconst_{\phi}+\|\infixset\|\cdot\|\phi\|+S_{n^-}\phi(\xi^-)<-\distconst_{\phi}.
        \end{equation*}
        Hence, we can take the smallest $k^-\in\set{1,\cdots,n^-}$ such that
        \begin{equation} \label{eqn:lma:gjk:proof}
            S_{|\omega|+|\rho^-|+k^-}\phi(\omega\rho^-\xi^-)<-\distconst_{\steplenpot_\alpha}.
        \end{equation}
        This means that $S_{|\omega|+|\rho^-|+k^--1}\phi(\omega\rho^-\xi^-)\geq-\distconst_{\steplenpot_\alpha}$, and therefore,
        \begin{align*}
            S_{|\omega|+|\rho^-|+k^-}\phi(\omega\rho^-\xi^-)&\geq S_{|\omega|+|\rho^-|+k^--1}\phi(\omega\rho^-\xi^-)-\|\phi\|\\
            &\geq-\distconst_{\phi}-\|\infixset\|\cdot\|\phi\|.
        \end{align*}
        Combining this fact with (\ref{eqn:lma:gjk:proof}), we can thus deduce that $\omega\rho^-\xi^-_1\cdots\xi^-_{k^-}\in\windowword_{\alpha,K}$, when $S_{|\omega|+|\rho^-|}\phi(\omega\rho^-\xi^-)>0$. We have thus shown that if $S_{|\omega|}\phi(\xi)\geq 0$, then either $\omega\rho^-$ or $\omega\rho^-\xi^-_1\cdots \xi^-_{k^-}$, for some $k^-\in\set{1,\cdots,n^-}$, is in $\windowword_{\alpha,K}$. This completes the proof for the case where $S_{|\omega|}\phi(\xi)\geq 0$. The case where $S_{|\omega|}\phi(\xi)<0$ can be treated in the same way.
    \end{proof}
\bibliographystyle{plain}
\bibliography{bdgibbs1}
\end{document}